\documentclass[reqno]{amsart}

\usepackage[utf8]{inputenc}
\usepackage{amsmath,amsthm,amssymb,stmaryrd,mathrsfs}

\usepackage{graphicx}
\usepackage{microtype}
\usepackage{xstring,nicefrac}\usepackage{refcount}\usepackage{lastpage}

\usepackage[margin=1in]{geometry}

\usepackage{setspace}
\newtheorem{Theorem}{Theorem}[section]
\newtheorem{Proposition}[Theorem]{Proposition} 
\newtheorem{Lemma}[Theorem]{Lemma}
\newtheorem{Corollary}[Theorem]{Corollary}
\newtheorem{Example}[Theorem]{Example}
\newtheorem*{Claim}{Claim}

\theoremstyle{definition}
\newtheorem{Definition}[Theorem]{Definition}

\newtheorem{Question}[Theorem]{Question}

\theoremstyle{remark}
\newtheorem{Remark}[Theorem]{Remark}

\DeclareMathOperator{\dom}{\mathsf{dom}}

\newcommand{\andd}{\wedge}

\newcommand{\la}{\langle}
\newcommand{\ra}{\rangle}
\newcommand{\da}{\!\downarrow}
\newcommand{\ua}{\!\uparrow}
\newcommand{\imp}{\rightarrow}

\newcommand{\Pb}{\mathbb{P}}
\newcommand{\Qb}{\mathbb{Q}}

\newcommand{\cs}{2^\omega}
\newcommand{\str}{2^{<\omega}}
\newcommand{\uh}{{\upharpoonright}}
\newcommand{\wtwor}{\mathsf{W2R}}
\newcommand{\bwtwor}{\mathsf{bW2R}}
\newcommand{\mlr}{\mathsf{MLR}}
\newcommand{\bmlr}{\mathsf{bMLR}}

\newcommand{\twomlr}{\mathsf{2MLR}}

\newcommand{\halts}{{\downarrow}}
\newcommand{\llb}{\llbracket}
\newcommand{\rrb}{\rrbracket}
\newcommand{\binrat}{\mathbb{Q}^+_2}

\newcommand{\emptystr}{\varepsilon}

\def§#1§{\text{#1}}% 

\def\jump{\emptyset'}
\def\str{2^{<\omega}}
\def\br#1{\overline{#1}}

\def\compl{^\mathsf{c}}

\newcommand{\MP}[1]{\mathcal{#1}}

\title{Randomness and Semi-measures}
\author{Laurent Bienvenu}
\address{Laboratoire CNRS J.-V. Poncelet, 119002, Bolshoy Vlasyevskiy Pereulok 11, Moscow, Russia}
\email{laurent.bienvenu@computability.fr}
\urladdr{www.liafa.univ-paris-diderot.fr/$\sim$lbienven}
\thanks{Laurent Bienvenu is supported by the ANR project ``EQINOCS" (ANR-11-BS02-004).}

\author{Rupert H\"olzl}
\address{Institut für Theoretische Informatik, Mathematik und
Operations Research, Fakultät für Informatik, Universität der Bundeswehr München, Werner-Heisenberg-Weg 39, 85577~Neubiberg, Germany}
\email{r@hoelzl.fr}
\urladdr{http://hoelzl.fr}
\thanks{Rupert Hölzl is supported by a Feodor Lynen postdoctoral research fellowship by the Alexander von Humboldt Foundation.}

\author{Christopher P. Porter}
\address{LIAFA\\
Universit\'e Paris Diderot - Paris 7\\
Case 7014\\
75205 Paris Cedex 13\\ 
France}
\email{cp@cpporter.com}
\urladdr{http://cpporter.com}
\thanks{Christopher Porter is supported by the National Science Foundation under grant OISE-1159158 as part of the International Research Fellowship Program.}

\author{Paul Shafer}
\address{Department of Mathematics\\
Ghent University\\
Krijgslaan 281, S22\\
B-9000 Ghent\\
Belgium}
\email{paul.shafer@ugent.be}
\urladdr{http://cage.ugent.be/$\sim$pshafer/}

\thanks{Paul Shafer is an FWO Pegasus Long Postdoctoral Fellow, and he also acknowledges the support of the Fondation Sciences Math\'ematiques de Paris.}


\date{\today}

\begin{document}
%\tableofcontents

\begin{abstract}
A semi-measure is a generalization of a probability measure obtained by relaxing the additivity requirement to super-additivity.  We introduce and study several randomness notions for left-c.e.\ semi-measures, a natural class of effectively approximable semi-measures induced by Turing functionals.  Among the randomness notions we consider, the generalization of weak $2$-randomness to left-c.e.\ semi-measures is the most compelling, as it best reflects Martin-L\"of randomness with respect to a computable measure.  Additionally, we analyze a question of Shen from \cite{BecBieDow12}, a positive answer to which would also have yielded a reasonable randomness notion for left-c.e.\ semi-measures. Unfortunately though, we find a negative answer, except for some special cases.
\end{abstract}

\maketitle

\section{Introduction}
Suppose we have an algorithmic procedure $P$ that, upon receiving an infinite binary sequence as an input, yields either an infinite binary sequence or a finite binary string as the output.  The question we investigate here is:
\begin{itemize}
\item[($\mathcal{Q}$)] What is the typical infinite output of $P$?
\end{itemize}
In the case that $P$ always produces an infinite output or produces an infinite output with probability one, there is already a complete answer to ($\mathcal{Q}$), insofar as we understand typicality in terms of Martin-L\"of randomness.  In this case, the typical outputs of $P$ are determined precisely by the behavior of $P$ on all random inputs: the procedure $P$ and the Lebesgue measure $\lambda$ together induce a measure $\lambda_P$ (in a sense to be made precise below) so that the typical infinite outputs of $P$ are exactly the sequences that are random with respect to the measure $\lambda_P$.

In this paper, we attempt to answer the question ($\mathcal{Q}$) in the case where the algorithmic procedure $P$ does not produce an infinite output with probability one.  Whereas an algorithmic procedure that yields an infinite output with probability one induces a computable measure, a procedure that yields an infinite output with probability less than one induces what is known as a \emph{left-c.e.\ semi-measure}, where a semi-measure is a function $\rho:\str\rightarrow[0,1]$ which satisfies
\begin{itemize}
\item[(i)] $\rho(\emptystr)=1$ and 
\item[(ii)] $\rho(\sigma)\geq\rho(\sigma0)+\rho(\sigma1)$,
\end{itemize}  
where $\emptystr$ denotes the empty string, and a semi-measure is left-c.e.\ if it is effectively approximable from below.

% and $\llb \tau \rrb$ denotes the set of infinite binary sequences beginning with the binary string $\tau$

As will be discussed in the next section, we can reformulate the question ($\mathcal{Q}$) as
\begin{itemize}
\item[($\mathcal{Q}'$)] Which infinite sequences are random with respect to the left-c.e.\ semi-measure induced by the Turing functional $\Phi$?
\end{itemize}
Clearly, answering this question requires a definition of randomness with respect to a left-c.e.\ semi-measure, but there is currently no such definition available.

One attempt to answering ($\mathcal{Q}'$) was suggested by Shen at a recent meeting at Dagstuhl, ``Computability, Complexity, and Randomness'' (see \cite{BecBieDow12}).  There he asked the following question, which was already raised in \cite{SheBieRom08}.
\begin{Question}\label{q:shen}
If $\Phi$ and $\Psi$ are Turing functionals that induce the same left-c.e.\ semi-measure, does it follow that $\Phi(\mlr)=\Psi(\mlr)$?
\end{Question}

\noindent The relevance of Shen's question to the task of defining randomness with respect to a left-c.e.\ semi-measure is this:  suppose Question \ref{q:shen} has a positive answer.  Then we can define $Y\in\cs$ to be $\rho$-random if and only if there is some $X\in\cs$ such that $\Phi(X)=Y$ for any $\Phi$ that induces~$\rho$.  We call this the \emph{push-forward} definition of randomness with respect to a semi-measure.  

In this paper, we show that Shen's question has a positive answer in the restricted case that $\Phi$ and $\Psi$ induce a \emph{computable} semi-measure.  Moreover, we show that if the definition of Martin-L\"of randomness for computable measures is extended to computable semi-measures, the resulting definition is equivalent to the push-forward definition of randomness with respect to a computable semi-measure.

The situation is much less straightforward when we consider left-c.e.\ semi-measures.  First, we show that Shen's question has a negative answer in this more general setting, and thus we need a different strategy for answering ($\mathscr{Q}'$).  Towards this end, we consider two general approaches to defining randomness with respect to a left-c.e.\ semi-measure:  
\begin{enumerate}
\item defining randomness with respect to a semi-measure by a direct adaptation of standard definitions of randomness with respect to computable measures, and 
\item defining randomness with respect to a semi-measure in terms of a specific measure derived from trimming back a given semi-measure to a measure.
\end{enumerate}
Although we prove a number of results about these candidate definitions, no definition has yet to emerge as the most well-behaved.  However, some of the results we present indicate that weak 2-randomness is a promising notion in this context.

The remainder of the paper is organized as follows.  In Section 2, we provide the necessary background on randomness with respect to computable and non-computable measures.  We also discuss some basic results about the relationship between semi-measures and Turing functionals.  In Section 3 we answer Shen's question when restricted to the collection of computable semi-measures by formulating a definition of randomness with respect to a computable semi-measure.  In Section 4, we answer the general version of Shen's question in the negative, but we do show that a related question involving a notion of randomness that is stronger than Martin-L\"of randomness has a positive answer. In Section 5, we pursue the first strategy for answering ($\mathcal{Q}'$) discussed above, directly modifying a number of different definitions of randomness with respect to a measure.  Lastly, in Section 6, we discuss the measure obtained by trimming back a semi-measure and explore the notions of randomness with respect to such measures.  

We assume that the reader is familiar with the basic notions from computability theory:  computable functions, partial computable functions, computably enumerable sets, Turing functionals, Turing degrees, and the Turing jump.  For details on algorithmic randomness, see~\cite{DowHir10} or ~\cite{Nie09}.

 Let us fix some notation and terminology. We denote by $\cs$ the set of infinite binary sequences, also known as Cantor space. We denote the set of finite strings by~$\str$ and the empty string by~$\emptystr$. $\binrat$ is the set of non-negative dyadic rationals, i.e., rationals of the form $m/2^n$ for $m,n\in\omega$.  Given $X\in \cs$ and an integer~$n$, $X \uh n$ is the string that consists of the first~$n$ bits of~$X$, and $X(n)$ is the $(n+1)^{\mathrm{st}}$ bit of $X$ (so that $X(0)$ is the first bit of $X$).  If $\sigma$ and $\tau$ are strings, then $\sigma\preceq\tau$ means that $\sigma$ is an initial segment of $\tau$.  Similarly for $X\in\cs$, $\sigma\preceq X$ means that $\sigma$ is an initial segment of $X$.  Given a string~$\sigma$, the \emph{cylinder} $\llb\sigma\rrb$ is the set of elements of $\cs$ having~$\sigma$ as an initial segment.  Similarly, given $S\subseteq\str$, $\llb S\rrb$ is defined to be the set $\bigcup_{\sigma\in S}\llb\sigma\rrb$. The cylinders form a basis for the usual topology on the Cantor space (the product topology), and thus the open sets for this topology are those of the form $\llb S \rrb$ for some~$S$. An open set $\mathcal{U}$ is said to be \emph{effectively open} (or $\Sigma^0_1$) if $\mathcal{U}=\llb S \rrb$ for some c.e.\ set of strings~$S$. An \emph{effectively closed set} (or $\Pi^0_1$) is the complement of an effectively open set. A sequence of open sets $(\mathcal{U}_n)_{n \in \omega}$ is said to be \emph{uniformly effectively  open} if there exists a sequence $(S_n)_{n \in \omega}$ of uniformly c.e.\ sets of strings such that $\mathcal{U}_n=\llb S_n \rrb$.

\section{Preliminaries}\label{sec-prelims}

In this section, we will review the basic results of randomness with respect to computable and non-computable measures.

\subsection{Randomness with respect to computable measures}

The standard definition of algorithmic randomness is Martin-L\"of randomness, first introduced by Martin-L\"of in \cite{Mar66}. Martin-L\"of's original definition was given in terms of the Lebesgue measure, but he also recognized that it held for a larger class of probability measures.  We first define Martin-L\"of randomness with respect to a computable measure $\mu$.

A measure $\mu$ on $\cs$ is \emph{computable} if $\sigma \mapsto \mu(\llb\sigma\rrb)$ is computable as a real-valued function, i.e., if there is a computable function $\tilde \mu:\str\times\omega\rightarrow\binrat$ such that
\[|\mu(\llb\sigma\rrb)-\tilde \mu(\sigma,i)|\leq 2^{-i}\]
for every $\sigma\in\str$ and $i\in\omega$.  From now on, we will write $\mu(\llb\sigma\rrb)$ as $\mu(\sigma)$.  By Caratheodory's Theorem, a Borel measure $\mu$ on $\cs$ is uniquely determined by the values $\mu(\sigma)$ for $\sigma\in\str$, and conversely, given a function $f: \str \in [0,1]$ such that $f(\emptystr)=1$ and $f(\sigma)=f(\sigma0)+f(\sigma1)$ for all~$\sigma$, there exists a (unique) Borel measure~$\mu$ such that $\mu(\llb\sigma\rrb)=f(\sigma)$ for all~$\sigma$. 

The \emph{uniform (or Lebesgue ) measure}~$\lambda$ is the measure for which each bit of the sequence has value~$0$ with probability $1/2$, independently of the values of the other bits. It can be defined as the unique Borel measure such that $\lambda(\sigma)=2^{-|\sigma|}$ for all strings~$\sigma$.

\begin{Definition}
Let $\mu$ be a computable measure on $\cs$.

\begin{itemize}
\item[(i)] A \emph{$\mu$-Martin-L\"of test} is a sequence $(\mathcal{U}_i)_{i\in\omega}$ of uniformly effectively open subsets of $\cs$ such that for each $i$,
\[
\mu(\mathcal{U}_i)\leq 2^{-i}.
\]

\item[(ii)] $X\in\cs$ passes the $\mu$-Martin-L\"of test $(\mathcal{U}_i)_{i\in\omega}$ if $X\notin\bigcap_{i \in \omega}\mathcal{U}_i$.

\medskip

\item[(iii)] $X\in\cs$ is \emph{$\mu$-Martin-L\"of random}, denoted $X\in\mlr_\mu$, if $X$ passes every $\mu$-Martin-L\"of test. When $\mu$ is the uniform measure $\lambda$, we often abbreviate $\mlr_\mu$ by $\mlr$.
\end{itemize}
\end{Definition}

An important feature of Martin-L\"of randomness is the existence of a universal test.  For every computable measure $\mu$, there is a universal $\mu$-Martin-L\"of test $(\hat{\MP{U}}_i)_{i\in\omega}$, having the property that $X\in\mlr_\mu$ if and only if $X\notin\bigcap_{i\in\omega}\hat{\MP{U}}_i$.

\begin{Definition}
For any computable measure~$\mu$, we tacitly assume that a universal $\mu$-Martin-L\"of test $(\hat{\MP{U}}_i)_{i\in\omega}$ has been fixed, and we denote by $\mlr_\mu^d$ the $\Pi^0_1$ set $(\hat{\MP{U}}_d)\compl$, so that $\mlr_\mu$ is the non-decreasing union of the sets $\mlr_\mu^d$.
\end{Definition}

Given a measure $\mu$, we say that $X\in\cs$ is an \emph{atom of $\mu$} if $\mu(\{X\})>0$.  Kautz proved the following useful fact about the atoms of a computable measure (see \cite{Kau91}).

\begin{Lemma}\label{lem-kautz}
$X\in\cs$ is computable if and only if $X$ is an atom of some computable measure.
\end{Lemma}

There is an intimate connection between Martin-L\"of random sequences and a class of effective functionals that induce computable measures, a connection that we would like to preserve when we formulate a definition of randomness with respect to a semi-measure.  Recall that a \emph{Turing functional} $\Phi:\subseteq\cs\rightarrow\cs$ may be defined as a c.e.\ set of pairs of strings $(\sigma,\tau)$ such that 
if $(\sigma,\tau),(\sigma',\tau')\in\Phi$ and $\sigma\preceq\sigma'$, then $\tau\preceq\tau'$ or $\tau'\preceq\tau$.  For each $\sigma\in\str$, we define $\Phi^\sigma$ to be the maximal string (for the prefix order) in $\{\tau: (\exists \sigma'\preceq\sigma)((\sigma',\tau)\in\Phi)\}$.  To obtain a map defined on $\cs$ from this c.e.\ set of pairs, for each $X\in\cs$, we let $\Phi^X$ be the maximal (for the prefix order) sequence~$Y$ such that $\Phi^{X \uh n}$ is a prefix of~$Y$ for all~$n$.   Note that $Y$ can be finite or infinite. We will thus set
$\dom(\Phi)=\{X\in\cs:\Phi^X\in\cs\}$.  When $\Phi^X\in\cs$, we will often write $\Phi^X$ as $\Phi(X)$ to emphasize the functional $\Phi$ as a map from~$\cs$ to~$\cs$.  
For $\tau \in \str$ let $\Phi^{-1}(\tau)$ be the set 
$\{ \sigma\in \str : \exists \tau' \succeq \tau \colon (\sigma,\tau')\in\Phi\}$.  Similarly, for $S \subseteq \str$ we define $\Phi^{-1}(S) = \bigcup_{\tau \in S} \Phi^{-1}(\tau)$. When $\mathcal{A}$ is a subset of $\cs$, we denote by $\Phi^{-1}(\mathcal{A})$ the set $\{X\in \dom(\Phi):\Phi(X) \in \mathcal{A}\}$. Note in particular that $\Phi^{-1}(\llb\tau\rrb)
= \llb\Phi^{-1}(\tau)\rrb \cap \dom(\Phi)$.

The Turing functionals that induce computable measures are precisely the \emph{almost total} Turing functionals, where a Turing functional $\Phi$ is almost total if 
\[
\lambda(\mathsf{dom}(\Phi))=1.
\]
Given an almost total Turing functional $\Phi$, the measure induced by $\Phi$, denoted $\lambda_\Phi$, is defined by
\[
\lambda_\Phi(\sigma)=\lambda(\llb\Phi^{-1}(\sigma)\rrb)
=\lambda(\{X:\Phi^X\succeq\sigma\}).
\]
It is not difficult to verify that $\lambda_\Phi$ is a computable measure.  Moreover, given a computable measure $\mu$, one can show that there is some almost total functional $\Phi$ such that $\mu=\lambda_\Phi$.

The following two results are very useful.  The first one, due to Levin and Zvonkin \cite{ZvoLev70} and known as the preservation of randomness theorem, says that randomness is preserved under almost total Turing functionals. The second one, due to Shen (unpublished, but see \cite{SheBieRom08}), is a partial converse of the preservation randomness theorem.  It says that sequences that are random with respect to some computable measure must have some unbiased random source.  We thus refer to this result as the \emph{No Randomness Ex Nihilo principle}, to reflect that one cannot produce randomness solely out of non-random sources.

\begin{Theorem} \label{thm:pres-ex-nihilo} Let $\Phi$ be an almost total Turing functional.
\begin{itemize}
\item[(i)]  Preservation of randomness: If $X\in\mlr$, $\Phi(X)\in\mlr_{\lambda_\Phi}$.
\item[(ii)] No Randomness Ex Nihilo principle:  If $Y\in\mlr_{\lambda_\Phi}$, there is some $X\in\mlr$ such that $\Phi(X)=Y$.
\end{itemize}
\end{Theorem}

It will be helpful to introduce several other notions of algorithmic randomness for computable measures.  First, the definition of Martin-L\"of randomness can be straightforwardly relativized to an oracle (for details, see ~\cite{DowHir10} or ~\cite{Nie09}).  In particular, for each $n$, if we relativize Martin-L\"of randomness to $\emptyset^{(n)}$, the $n^{\mathrm{th}}$ jump of the empty set, this yields a notion known as $(n+1)$-randomness.  Another definition of randomness we consider is weak 2-randomness.

\begin{Definition}
Let $\mu$ be a computable measure.

\begin{itemize}
\item[(i)] A \emph{generalized $\mu$-Martin-L\"of test} is a sequence $(\mathcal{U}_i)_{i\in\omega}$ of uniformly $\Sigma^0_1$ subsets of $\cs$ such that
\[
\lim_{i\rightarrow\infty}\mu(\mathcal{U}_i)=0.
\]

\item[(ii)] $X\in\cs$ passes the generalized $\mu$-Martin-L\"of test $(\mathcal{U}_i)_{i\in\omega}$ if $X\notin\bigcap_{i \in \omega}\mathcal{U}_i$.

\medskip

\item[(iii)] $X\in\cs$ is \emph{$\mu$-weakly 2-random}, denoted $X\in\wtwor_\mu$, if $X$ passes every generalized $\mu$-Martin-L\"of test. When $\mu$ is the uniform measure $\lambda$, we often abbreviate $\wtwor_\mu$ by $\wtwor$.
\end{itemize}
\end{Definition}

As every $\mu$-Martin-L\"of test is a generalized $\mu$-Martin-L\"of test, it follows that $\wtwor_\mu\subseteq\mlr_\mu$.  In general, the converse does not hold, as shown by the following result.  Recall that $A,B\in\cs$ form a Turing minimal pair if for any $C\in\cs$, $C\leq_T A$ and $C\leq_T B$ implies that $C\equiv_T \emptyset$.

\begin{Theorem}\label{thm-w2r-minimal-pair}
Let $\mu$ be a computable measure on $\cs$.  If $X\in\cs$ is not computable, then $X$ is $\mu$-weakly 2-random if and only if $X$ is $\mu$-Martin-L\"of random and forms a Turing minimal pair with $\emptyset'$.
\end{Theorem}

The proof of this theorem is a generalization of the proof of the result in the case that $\mu$ is the Lebesgue measure (see the proof of Theorem 2.69 of \cite{Por12} for details).  One direction of the original theorem was proved in \cite{DowNieWeb06}, while the other direction was proved by Hirschfeldt and Miller (unpublished; see Theorem 7.2.11 of \cite{DowHir10}).

\subsection{Randomness with respect to non-computable measures}\label{subsec-non-computable-measures}

Let $\mathscr{P}(\cs)$ be the collection of probability measures on $\cs$.  It can be equipped with a natural topology, the so-called weak topology. The set $\mathcal{B}$ of subsets of $\mathscr{P}(\cs)$ of type
\[
\left\{\mu : \bigwedge_{i=1}^n [\ell_i < \mu(\sigma_i) < r_i] \right\}
\]
where the $\sigma_i$ are strings and the $\ell_i, r_i$ are rational numbers, form a base for this topology (note that such sets can be encoded by an integer, and we call $B_i$ the set of code~$i$). 

We will consider two general approaches to defining randomness for a non-computable measure $\mu\in\mathscr{P}(\cs)$, depending on whether our test has access to the measure as an oracle.  If we allow our test to have access to the measure as an oracle, we first need to code it as an infinite binary sequence. For this we fix a surjective partial map $\Theta:\cs\rightarrow\mathscr{P}(\cs)$, defined on a $\Pi^0_1$-subset of $\cs$, which must have the following property: from every enumeration of $X \in \dom(\Theta)$ (seen as a subset of $\omega$), one can uniformly enumerate the $B_i$'s containing $\mu$, and from any enumeration of the $B_i$'s containing~$\mu$ one can uniformly enumerate some pre-image of $\mu$ by $\Theta$. We say that an enumeration of~$X$ is a \emph{representation} of $\Theta(X)$. \label{page:representations} There are a number of equivalent ways to carry this out (see \cite{Rei08} or \cite{DayMil13}), the easiest one being to define $\Theta(X)$ to be the measure (if it exists and is unique) contained in $B_i$ for each $i \in X$.

The important caveat is that there are measures such that, among all their representations, there is none of smallest Turing degree (this follows for example from the existence of a neutral measure, as shown by Levin \cite{MR0429327}), a phenomenon which occurs no matter what particular representation is chosen. Therefore, there is no canonical way to represent a measure by a unique member of $\cs$, and subsequently, in any definition where one wants to treat $\mu$ as an oracle, one needs to quantify over representations of~$\mu$.

\begin{Definition}
Let $\mu$ be a measure on $\cs$, and let $R$ be a representation of $\mu$.

\begin{itemize}
\item[(i)] An \emph{$R$-Martin-L\"of test} is a sequence $(\mathcal{U}_i)_{i\in\omega}$ of uniformly $\Sigma^0_1(R)$ subsets of $\cs$ such that for each $i$,
\[
\mu(\mathcal{U}_i)\leq 2^{-i}.
\]
\item[(ii)] $X\in\cs$ passes the $R$-Martin-L\"of test $(\mathcal{U}_i)_{i\in\omega}$ if $X\notin\bigcap_{i \in \omega}\mathcal{U}_i$.
\item[(iii)] $X\in\cs$ is \emph{$R$-Martin-L\"of random}, denoted $X\in\mlr_\mu^R$, if $X$ passes every $R$-Martin-L\"of test.
\item[(iv)] $X\in\cs$ is \emph{$\mu$-Martin-L\"of random}, denoted $X\in\mlr_\mu$, if there is some representation $R$ of $\mu$ such that $X$ is $R$-Martin-L\"of random.
\end{itemize}
\end{Definition}

An alternative approach to defining randomness with respect to a non-computable measure dispenses with the representations, resulting in what is known as \emph{blind randomness} (or \emph{Hippocratic randomness}, as it was called by Kjos-Hanssen in \cite{Kjo10}, where the definition first appeared).

\begin{Definition}
Let $\mu$ be a measure on $\cs$.

\begin{itemize}
\item[(i)] A \emph{blind $\mu$-Martin-L\"of test} is a sequence $(\mathcal{U}_i)_{i\in\omega}$ of uniformly $\Sigma^0_1$ (i.e.\ effectively open) subsets of $\cs$ such that for each $i$,
\[
\mu(\mathcal{U}_i)\leq 2^{-i}.
\]
\item[(ii)] $X\in\cs$ passes the blind $\mu$-Martin-L\"of test $(\mathcal{U}_i)_{i\in\omega}$ if $X\notin\bigcap_{i \in \omega}\mathcal{U}_i$.
\item[(iii)] $X\in\cs$ is \emph{blind $\mu$-Martin-L\"of random}, denoted $X\in\bmlr_\mu$, if $X$ passes every blind $\mu$-Martin-L\"of test.
\end{itemize}
\end{Definition}

\subsection{Some basic facts about left-c.e.\ semi-measures} 

Recall from the introduction that a semi-measure $\rho:\str\rightarrow[0,1]$ satisfies
\begin{itemize}
\item[(i)] $\rho(\emptystr) = 1$ and 
\item[(ii)] $\rho(\sigma)\geq\rho(\sigma0)+\rho(\sigma1)$.
\end{itemize}  
Henceforth, we will restrict our attention to the class of left-c.e.\ semi-measures, where a semi-measure $\rho$ is \emph{left-c.e.}\ if, uniformly in $\sigma$, there is a computable function $\tilde \rho : \str \times \omega \rightarrow \binrat$, non-decreasing in its first argument, and such that for all~$\sigma$:
\[
\lim_{i \rightarrow +\infty} \tilde \rho(\sigma,i) = \rho(\sigma)
\]
That is, the values of $\rho$ on basic open sets are uniformly approximable from below.  

Just as computable measures are precisely the measures that are induced by almost total Turing functionals, left-c.e.\ semi-measures are precisely the semi-measures that are induced by Turing functionals:

\begin{Theorem}[Levin, Zvonkin \cite{ZvoLev70}]\label{thm-MachinesInduceSemiMeasures}{\ }
\begin{itemize}
\item[(i)] For every Turing functional $\Phi$, the function $\lambda_\Phi(\sigma)=\lambda(\llb\Phi^{-1}(\sigma)\rrb)
=\lambda(\{X:\Phi^X\succeq\sigma\})$ is a left-c.e.\ semi-measure, and  
\item[(ii)] for every left-c.e.\ semi-measure $\rho$, there is a Turing functional $\Phi$ such that $\rho=\lambda_\Phi$. 
\end{itemize}
\end{Theorem}

Another significant fact about left-c.e.\ semi-measures is the existence of a \emph{universal} left-c.e.\ semi-measure:  there exists a left-c.e.\ semi-measure $M$ such that, for every left-c.e.\ semi-measure~$\rho$, there exists a $c\in\omega$ such that $\rho\leq c\cdot M$.  One way to obtain a universal left-c.e.\ semi-measure is to effectively list all left-c.e.\ semi-measures $(\rho_e)_{e \in \omega}$ (which can be obtained from an effective list of all Turing functionals by appealing to Theorem~\ref{thm-MachinesInduceSemiMeasures}) and set $M = \sum_{e \in \omega} 2^{-e-1} \rho_e$.  Alternatively, one can induce it by means of a universal Turing functional.  Let $(\Phi_i)_{i\in\omega}$ be an effective enumeration of all Turing functionals.  Then the functional $\widehat\Phi$ such that
\[
\widehat\Phi(1^e0X)=\Phi_e(X)
\]
for every $e\in\omega$ and $X\in\cs$ is a universal Turing functional and we can set $M=\lambda_{\widehat\Phi}$. Then $M$ is an universal left-c.e.\ semi-measure, since for any left-c.e.\ semi-measure $\rho$, there is some $\Phi_e$ such that $\rho=\lambda_{\Phi_e}$, and thus by definition of $\widehat\Phi$, we have $\lambda_{\Phi_e} \leq 2^{e+1}\cdot\lambda_{\widehat\Phi}$.

\section{Shen's Question for Computable semi-measures}\label{sec-computable-semi-measures}

In this section, we provide a positive answer to Shen's question for the case of computable measures.  That is, we prove:

\begin{Theorem}\label{thm-shen-question-computable}
If $\Phi$ and $\Psi$ are Turing functionals such that $\lambda_\Phi=\lambda_\Psi$ and $\lambda_\Phi$ is computable, then $\Phi(\mlr)=\Psi(\mlr)$.  
\end{Theorem}

To prove Theorem \ref{thm-shen-question-computable}, we extend the definition of Martin-L\"of randomness with respect to computable measures to a definition of Martin-L\"of randomness with respect to computable semi-measures.

The definition of a computable semi-measure is just a slight modification of the definition of a computable measure:  a semi-measure $\rho$ is \emph{computable} if there is a computable function $\tilde\rho:\str\times\omega\rightarrow\binrat$ such that
\[
|\rho(\sigma)-\tilde\rho(\sigma,i)|\leq 2^{-i}
\]
for every $\sigma\in\str$ and $i\in\omega$ .

To define Martin-L\"of randomness with respect to a computable semi-measure, we 
have to exercise some caution.  In general, for a given $\Sigma^0_1$ class $\MP U$, $\rho(\MP U)$ is not well-defined, as there may exist prefix-free sets $E_0, E_1 \subseteq 2^{<\omega}$ such that $\llb E_0\rrb = \llb E_1\rrb = \MP U$ but $\rho(E_0) \neq \rho(E_1)$, if one defines $\rho(E)=\sum_{\sigma\in E}\rho(\sigma)$ for $E\subseteq\str$.

To remedy this problem, we will only apply semi-measures to c.e.\ subsets of $\str$ rather than to effectively open subsets of $\cs$.  Moreover, since any c.e.\ set $E \subseteq 2^{<\omega}$ may be replaced with a prefix-free c.e.\ set $F \subseteq 2^{<\omega}$ such that $\llb F\rrb = \llb E\rrb$ and $\rho(F) \leq \rho(E)$ for any semi-measure $\rho$, we can always assume that a given c.e.\ subset of $\str$ is prefix-free.  This replacement can be done uniformly, so whenever we need to consider a uniformly c.e.\ sequence $(E_i)_ {i\in\omega}$ of subsets of $2^{<\omega}$, we may assume that the sets $E_i$ are all prefix-free.

\begin{Definition}
Let $\rho$ be a computable semi-measure.

\begin{itemize}
\item[(i)] A \emph{$\rho$-Martin-L\"of test} is a uniformly c.e.\ sequence $(U_i)_{i\in\omega}$ of subsets of $\str$ such that 
\[
\rho(U_i)\leq 2^{-i}.
\]
for each $i\in\omega$.
\item[(ii)] $X\in\cs$ passes the $\rho$-Martin-L\"of test $(U_i)_{i\in\omega}$ if $X\notin\bigcap_{i \in \omega}\llb U_i\rrb$.

%\medskip

\item[(iii)] $X\in\cs$ is \emph{$\rho$-Martin-L\"of random}, denoted $X\in\mlr_\rho$, if $X$ passes every $\rho$-Martin-L\"of test.
\end{itemize}
\end{Definition}

We now verify that randomness with respect to a \emph{computable} semi-measure satisfies both randomness preservation and the No Randomness Ex Nihilo principle.  Whereas the proof of randomness preservation for computable measures is essentially the same as the standard proof of randomness preservation for computable measures, the proof of the No Randomness Ex Nihilo principle for computable semi-measures is considerably more delicate than the original proof.

\begin{Theorem}[Randomness preservation for computable semi-measures]\label{thm-preservation-comp-sem}
If $\Phi$ is a Turing functional that induces a computable semi-measure $\rho$, then $X\in\mlr\cap\dom(\Phi)$ implies $\Phi(X)\in\mlr_\rho$.
\end{Theorem}

\begin{proof}
Suppose that we have $X\in\dom(\Phi)$ such that $\Phi(X)\notin\mlr_\rho$.  Then there is a $\rho$-Martin-L\"of test $(U_i)_{i\in\omega}$ such that $\Phi(X)\in\bigcap_{i \in \omega}\llb U_i\rrb$.  We define
\[
\mathcal{V}_i=\bigcup_{\tau\in U_i}\llb\Phi^{-1}(\tau)\rrb.
\]
Clearly, the collection $(\mathcal{V}_i)_{i\in\omega}$ is uniformly $\Sigma^0_1$.  Then 
\[
\lambda(\mathcal{V}_i)\leq\sum_{\tau\in U_i}\lambda(\llb\Phi^{-1}(\tau)\rrb)=\sum_{\tau\in U_i}\rho(\tau)=\rho(U_i)\leq 2^{-i},
\]
so $(\mathcal{V}_i)_{i\in\omega}$ is a Martin-L\"of test.  Lastly, $\Phi(X)\in \llb U_i\rrb$ for each $i\in\omega$, thus for each $i \in \omega$ there is some $\tau\preceq\Phi(X)$ such that $\tau\in U_i$. This implies that $X\in\llb\Phi^{-1}(\tau)\rrb$, and so we have $X\in\mathcal{V}_i$.  Thus, $X\notin\mlr$.
\end{proof}

\begin{Theorem}[No Randomness Ex Nihilo principle for computable semi-measures]\label{thm-ex-nihilo-comp-sem}
Let $\Phi$ be a Turing functional that induces a computable semi-measure $\rho$.  If $Y\in\mlr_\rho$, then there is some $X\in\mlr$ such that $\Phi(X)=Y$.
\end{Theorem}

\begin{proof}

First we define a collection of Turing functionals $(\widehat\Phi_e)_{e\in\omega}$ that will serve as approximations for the functional $\Phi$.  Note that $\dom(\Phi)=\bigcap_{\ell\in\omega}\MP S_\ell$, where for each~$\ell$,
\[
\MP S_\ell=\{X\in\cs:(\exists k)\;|\Phi^{X\uh k}|\geq \ell\}.
\]
(which is uniformly effectively open).

For each $e$, we define a sequence of finite sets of strings $(C^e_\ell)_{\ell\in\omega}$ such that for every $\ell$
\begin{itemize}

\item[(i)] $\llb C_\ell^e \rrb \subseteq \MP S_\ell$, and
\item[(ii)] $\lambda(\MP S_\ell\setminus \llb C^e_\ell\rrb)\leq 2^{-\ell-e-1}$.
\end{itemize}
The sequence $(C^e_\ell)_{\ell\in\omega}$ can be effectively obtained, since $\rho$ is a computable semi-measure that is induced by $\Phi$, which implies that $\lambda(\MP S_\ell)$ is computable uniformly in~$\ell$. Each $\llb C^e_\ell \rrb$ is clopen, and therefore so are the sets $\bigcap_{k \leq \ell} \llb C^e_k \rrb$ for each $\ell\in\omega$. Let then $(D^e_\ell)_{e,\ell\in\omega}$ be a computable bi-sequence of sets of finite strings such that
\[
\llb D^e_\ell \rrb =\bigcap_{k \leq \ell} \llb C^e_k \rrb.
\]

Next we set 
\[
\widehat\Phi_e:=\{(\sigma,\tau)\in\Phi: \sigma\in D^e_{|\tau|}\},
\]
so that 
\[
\dom(\widehat\Phi_e)=\bigcap_{\ell\in\omega}\llb D^e_\ell\rrb =\bigcap_{\ell\in\omega}\llb C^e_\ell\rrb.
\]
Since each $\llb D^e_\ell\rrb$ is clopen, it follows that $\dom(\widehat\Phi_e)$ is a $\Pi^0_1$ class uniformly in~$e$. Moreover, $\Phi_e$ is a restriction of~$\Phi$ such that
\[
\lambda ( \dom(\Phi) \setminus \dom(\Phi_e) ) \leq \sum_{\ell\in\omega}\lambda(\MP S_\ell \setminus \llb C^e_\ell \rrb) \leq \sum_{\ell\in\omega} 2^{-e-l-1} \leq 2^{-e}.
\]
Note also that for for each $e,\ell$
\[
\lambda(\MP S_\ell \setminus \llb D^e_\ell \rrb) = \lambda \left(\bigcap_{k \leq \ell} \MP S_k \setminus \bigcap_{k \leq \ell} \llb C^e_k \rrb \right) \leq \sum_{\ell \leq k} 2^{-e-\ell-1} \leq 2^{-e}
\]
(where we use the definition of the $D^e_\ell$ and the fact that the $S_k$ are non-increasing), an inequality we will need at the end of the proof.

Now, for each $e\in\omega$, let $\Theta_e$ be the predicate on $\str$ defined by
\[
\Theta_e(\tau)\;\;\text{if and only if}\;\; \forall X [X\notin\mlr^e \;\vee\; X \notin \llb D^e_{|\tau|} \rrb \;\vee\; \Phi^X \bot \tau],
\]
where $ \Phi^X \bot \tau$ means that $ \Phi^X$ has length at least $|\tau|$ and is incomparable with~$\tau$, and $\mlr^e$ is the complement of the $e^{\mathrm{th}}$ level of the universal Martin-L\"of test (with respect to the Lebesgue measure).  The predicate $[X\notin\mlr^e \;\vee\; X \notin \llb D^e_{|\tau|} \rrb \;\vee\; \Phi^X \bot \tau]$ is $\Sigma^0_1$ over~$X$; therefore, by effective compactness, $\Theta_e$ is also $\Sigma^0_1$ uniformly in~$e$. For each~$e$, let $V_e$ be a maximal prefix-free set of strings among those satisfying $\Theta_e$. Note that $V_e$ is effectively open uniformly in~$e$. Let us evaluate $\lambda_\Phi(V_e)$:
\begin{eqnarray*}
\lambda_\Phi(V_e)&  = & \lambda(\{ X : (\exists \tau \in V_e) \; \Phi^X \succeq \tau\})\\
 & \leq & \lambda(\{ X : (\exists \tau) \; \Phi^X \succeq \tau \wedge(X \notin \llb D^e_{|\tau|} \rrb \vee X \in (\mlr^e)\compl)\})\\
 & \leq & \lambda \left(\bigcup_l \MP S_{l} \setminus \llb D^e_{l}\rrb \right) + \lambda((\mlr^e)\compl) \\
 & \leq & 2^{-e} + 2^{-e}.
\end{eqnarray*}

Thus, $(V_e)_{e\in\omega}$ is a $\lambda_\Phi$-Martin-L\"of test. This means that for every $\lambda_\Phi$--Martin-L\"of random~$Y$, there must be an~$e$ such that $Y \notin \llb V_e \rrb$, or in other words (by definition of $V_e$): for every prefix $Y \uh \ell$, there is some $X_\ell \in \mlr^e \cap \llb D^e_\ell \rrb$ such that $\widehat\Phi_e^{X_\ell} \succeq Y \uh \ell$. By compactness, one can assume, up to extraction of a subsequence, that the sequence $(X_\ell)_{\ell\in\omega}$ converges to some $X^*$. Since $X_\ell \in \llb D^e_\ell \rrb$ for all~$\ell$, and since the sets $\llb D^e_\ell \rrb$ are closed and non-increasing, it follows that $X^*$ belongs to all $\llb D^e_\ell \rrb$, i.e., $X^*$ is in the domain of $ \widehat\Phi_e$. By continuity of Turing functionals, $\widehat\Phi_e^{X^*} = \lim_\ell \widehat\Phi_e^{X_\ell} = \lim_\ell Y \uh \ell = Y$. Moreover, each $X_\ell$ belongs to $\mlr^e$ and $\mlr^e$ is closed, so $X^*$ belongs to $\mlr^e$ as well. Therefore~$Y$ has a Martin-L\"of random pre-image by $\widehat\Phi_e$, namely~$X^*$. Since $\widehat\Phi_e$ is a restriction of $\Phi$, the result follows.
\end{proof}

We are now in a position to prove Theorem \ref{thm-shen-question-computable}.

\begin{proof}[Proof of Theorem \ref{thm-shen-question-computable}]
Given $Y\in\Phi(\mlr)$, it follows from Theorem \ref{thm-preservation-comp-sem} that $Y\in\mlr_\rho$.  Since $\Psi$ induces $\rho$, by Theorem \ref{thm-ex-nihilo-comp-sem} there is some $X\in\mlr$ such that $\Psi(X)=Y$.  This shows that $\Phi(\mlr)\subseteq\Psi(\mlr)$, and a symmetric argument shows that $\Psi(\mlr)\subseteq\Phi(\mlr)$.
\end{proof}

Notice that the positive answer to Shen's question is an immediate consequence of randomness preservation and the No Randomness Ex Nihilo principle.  We see this again in Corollary~\ref{cor-shen-question-2ran} below in the context of left-c.e.\ semi-measures and $2$-randomness.

As our definition of randomness with respect to a computable semi-measure behaves much like Martin-L\"of randomness with respect to a computable measure, it is reasonable to ask if there are any sequences that are random with respect to some computable semi-measure but no computable measure.  We answer this question in the negative.

\begin{Proposition}
$X\in\cs$ is random with respect to a computable measure if and only if $X$ is random with respect to a computable semi-measure.
\end{Proposition}

\begin{proof}
As every computable measure is a computable semi-measure, one direction is immediate.  Suppose now that $X$ is not random with respect to any computable measure.  Let $\rho$ be a computable semi-measure.  We define the function $g:\str\rightarrow[0,1]$ to be
\[
g(\sigma)=\rho(\sigma)-(\rho(\sigma0)+\rho(\sigma1))
\]
for every $\sigma\in\str$.  Clearly $g$ is computable since $\rho$ is.  Next we define $\mu:\str\rightarrow[0,1]$ so that $\mu(\emptystr)=1$ and for $|\sigma| \geq 1$:
\[
\mu(\sigma)=\rho(\sigma)+\sum_{\tau\prec\sigma}2^{|\tau|-|\sigma|}g(\tau).
\]
Clearly $\mu$ is computable and $\rho(\sigma)\leq\mu(\sigma)$ for every $\sigma\in\str$.  We just need to verify that $\mu$ is a measure, which we prove by induction.  For any $\sigma\in\str$
\[
\begin{split}
\mu(\sigma0)+\mu(\sigma1)&=\rho(\sigma0)+\sum_{\tau\prec\sigma0}2^{|\tau|-|\sigma0|}g(\tau)+\rho(\sigma1)+\sum_{\tau\prec\sigma1}2^{|\tau|-|\sigma1|}g(\tau)\\
&=\rho(\sigma0)+\frac{1}{2}\sum_{\tau\prec\sigma0}2^{|\tau|-|\sigma|}g(\tau)+\rho(\sigma1)+\frac{1}{2}\sum_{\tau\prec\sigma1}2^{|\tau|-|\sigma|}g(\tau)\\
&=\rho(\sigma0)+\rho(\sigma1)+\sum_{\tau\preceq\sigma}2^{|\tau|-|\sigma|}g(\tau)\\
&=\rho(\sigma0)+\rho(\sigma1)+g(\sigma)+\sum_{\tau\prec\sigma}2^{|\tau|-|\sigma|}g(\tau)\\
&=\rho(\sigma)+\sum_{\tau\prec\sigma}2^{|\tau|-|\sigma|}g(\tau)\\
&=\mu(\sigma).
\end{split}
\]
Now since $X\notin\mlr_\mu$ by hypothesis, there is some $\mu$-Martin-L\"of test $(\mathcal{U}_i)_{i\in\omega}$ such that $X\in\bigcap_{i\in\omega}\mathcal{U}_i$.  Letting $U_i$ be such that $\llb U_i\rrb=\mathcal{U}_i$ for each $i\in\omega$,  $\rho(U_i)\leq\mu(\mathcal{U}_i)$ for every $i$, which implies that $(U_i)_{i\in\omega}$  is a $\rho$-Martin-L\"of test.  Thus $X\notin\mlr_\rho$.
\end{proof}

\section{Shen's question for left-c.e.\ semi-measures}

In this section, we prove that Question \ref{q:shen}, Shen's original question for left-c.e.\ semi-measures, has a negative answer.

\begin{Theorem}\label{thm-shen-counterexample}
There exist Turing functionals $\Phi$ and $\Psi$ such that $\lambda_\Phi=\lambda_\Psi$ and yet $\Phi(\mlr)\neq\Psi(\mlr).$

\end{Theorem}

\begin{proof}
We define $\Phi$ and $\Psi$ as c.e.\ sets of pairs $(\sigma,\tau)\in\str\times\str$.  Recall that Chaitin's $\Omega$ is defined to be
\[
\Omega:=\sum_{U(\sigma)\halts}2^{-|\sigma|},
\]
where $U$ is a universal prefix-free Turing machine (see~\cite{DowHir10} or ~\cite{Nie09} for more details).  Further, it is well-known that $\Omega$ is Martin-L\"of random and left-c.e.  Let $(\Omega_s)_{s\in\omega}$ be a computable non-decreasing sequence of rationals converging to $\Omega$.  We can think of each $\Omega_s$ as a finite string, so that for $n<|\Omega_s|$, $\Omega_s(n)$ is the $n^{\mathrm{th}}$ bit of the string $\Omega_s$. 

We define the following functional
\[
\Phi = \bigcup_n \{(\Omega_s \uh n , 0^n) : s \geq n\} 
\]

It is easy to see that $\dom(\Phi)=\{\Omega\}$, and $\Phi(\Omega)=0^\omega$ . Indeed, if $X \not= \Omega$, $X$ and $\Omega$ disagree on some bit, say the $k$-th bit, and then for some $t$ we have, for all $s \geq t$, $\Omega_s \uh k = \Omega \uh k \not= X \uh k$ and thus by construction $|\Phi^X|<t$, i.e. , $X \notin \dom(\Phi)$.

Next, we define $\Psi$.  For each $(\sigma,0^{|\sigma|})$ that we enumerate into $\Phi$ at stage $s$, let $\tau$ be the leftmost string of length $|\sigma|$ such that $(\tau,0^{|\sigma|})$ has not yet been enumerated into $\Psi$ and enumerate this pair into $\Psi$.  Observe that this construction ensures that (1) for all~$n$, $\lambda_\Phi(0^n)=\lambda_\Psi(0^n)$, and thus $\lambda_\Phi=\lambda_\Psi$ as both are equal to~$0$ on strings that are not of type~$0^n$ and (2) the domain of~$\Psi$ contains $0^\omega$ and is closed downards under the lexicographic order. A set which is closed downwards under the lexicographic order is either the empty set, the singleton~$0^\omega$, or a set of positive measure. It is not the emptyset and it cannot have positive measure, because otherwise there would exists a positive~$r$ such that $\lambda_\Psi(0^n) > r$ for all~$n$. This is impossible since $\lambda_\Psi=\lambda_\Phi$ and $\lambda_\Phi(0^n)$ tends to~$0$. Thus, $\dom(\Psi)=\{0^\omega\}$ and $\Psi(0^\omega)=0^\omega$, which in particular implies $\Psi(\mlr)=\emptyset\not=\Phi(\mlr)$.
\end{proof}

\begin{Remark}
The above proof actually works for any $\Delta^0_2$ Martin-L\"of random sequence.  Further, it is not necessary that $\lambda(\dom(\Phi))=0$.  If we define $\widehat\Phi$ and $\widehat\Psi$ by
\[
\left\{ \begin{array}{l} \widehat{\Phi}(0X)=\Phi(X) \\ \widehat\Phi(1X)=X  \end{array} \right. ~  ~  ~  ~  \text{and} ~  ~  ~  ~ ~ \left\{ \begin{array}{l} \widehat{\Psi}(0X)=\Psi(X) \\ \widehat\Psi(1X)=X  \end{array} \right.
\]

(where $\Phi$ and $\Psi$ are defined in the previous proof) we then have $\lambda(\dom(\widehat\Phi))=\lambda(\dom(\widehat\Psi))=\frac{1}{2}$, while $
\lambda_{\widehat\Phi}=(\lambda_\Phi+\lambda)/2=(\lambda_\Psi+\lambda)/2=\lambda_{\widehat\Psi}
$, and $\widehat\Phi(0\Omega)=\Phi(\Omega)$ has no Martin-L\"of random pre-image via $\widehat\Psi$. 
\end{Remark}

Although Question \ref{q:shen} has a negative answer, if we rephrase the question in terms of a stronger notion of randomness then we can answer the question in the affirmative.  To do so, we have to extend our definition of randomness with respect to a computable semi-measure to a definition of 2-randomness with respect to a $\emptyset'$-computable semi-measure.

First, we extend several definitions from the previous section:
\begin{itemize}
\item A semi-measure $\rho$ is $\jump$-computable if the values $(\rho(\sigma))_{\sigma\in\str}$ are uniformly $\jump$-computable.  
\item For a $\emptyset'$-computable semi-measure $\rho$, a $\rho$-$\jump$-Martin-L\"of test is a uniformly $\jump$-c.e.\ sequence $(U_i)_{i\in\omega}$ of subsets of $\str$ such that $\rho(U_i)\leq 2^{-i}.$
\item A sequence $X\in\cs$ passes the $\rho$-$\jump$-Martin-L\"of test $(U_i)_{i\in\omega}$ if $X\notin\bigcap_{i \in \omega}\llb U_i\rrb$.
\item For a $\jump$-computable semi-measure $\rho$, $X\in\cs$ is $\rho$-2-random, denoted $X\in\twomlr_\rho$, if $X$ passes every $\rho$-$\jump$-Martin-L\"of test.
\end{itemize}

\noindent Using these definitions, the following result is obtained from relativizing the proof of Theorems \ref{thm-preservation-comp-sem} and \ref{thm-ex-nihilo-comp-sem}.

\begin{Corollary}\label{cor-2-randoms-rand-pres-ex-nihilo}
Let $\rho$ be a left-c.e.\ semi-measure, and let $\Phi$ be a Turing functional such that $\rho=\lambda_\Phi$.
\begin{itemize}
\item[(i)] For every $X\in\twomlr\cap\dom(\Phi)$, $\Phi(X)\in\twomlr_\rho$.
\item[(ii)] If $Y\in\twomlr_\rho$, then there is some $X\in\twomlr$ such that $\Phi(X)=Y$.
\end{itemize}
\end{Corollary}

\noindent Corollary \ref{cor-2-randoms-rand-pres-ex-nihilo}, together with an argument similar to the one in the proof of Theorem \ref{thm-shen-question-computable}, yields the following.

\begin{Corollary}\label{cor-shen-question-2ran}
If $\Phi$ and $\Psi$ are Turing functionals such that $\lambda_\Phi=\lambda_\Psi$, then $\Phi(\twomlr)=\Psi(\twomlr)$.  
\end{Corollary}

\section{The direct adaptation approach}\label{sec-direct-modifications}

A positive answer to Shen's question would have yielded a definition of randomness with respect to a left-c.e.\ semi-measure:  for a left-c.e.\ semi-measure $\rho$, the sequences that are random with respect to $\rho$ would simply be the images of the Martin-L\"of random sequences under any functional that induces~$\rho$.  But as we have answered Shen's question in the negative, we need a different strategy to define randomness with respect to a left-c.e.\ semi-measure.  

In this section, we discuss certain desiderata for our definition and then we consider several definitions of randomness with respect to a left-c.e.\ semi-measure that are obtained by directly modifying standard definitions of randomness with respect to a computable measure.

\subsection{Desiderata for a definition of randomness with respect to a left-c.e.\ semi-measure}\label{subsec-desiderata}

Given that the collection of left-c.e.\ semi-measures extends the collection of computable measures, we would like our theory of randomness with respect to a left-c.e.\ semi-measure to extend the standard theory of randomness with respect to a computable measure.  To this end, it would be ideal to find a definition of randomness with respect to a semi-measure that satisfies a number of conditions, which we describe below. 

First, as every computable measure is a left-c.e.\ semi-measure, it seems natural to require the following:
\begin{itemize}
\item[(i)] {\bf Coherence:} $X$ is random with respect to a computable measure $\mu$ if and only if $X$ is random with respect to $\mu$ considered as a left-c.e.\ semi-measure.
\end{itemize}
Second, as the relationship between almost total Turing functionals and computable measures is analogous to the relationship between Turing functionals and left-c.e.\ semi-measures, we would like to extend the analogy by requiring the following two conditions:
\begin{itemize}
\item[(ii)] {\bf Randomness Preservation:} If $X$ is random and $\Phi$ is a Turing functional, then $\Phi(X)$ is random with respect to the semi-measure $\lambda_\Phi$.
\item[(iii)] {\bf No Randomness Ex Nihilo Principle:} If $Y$ is random with respect to the semi-measure $\lambda_\Phi$ for some Turing functional $\Phi$, then there is some random $X$ such that $\Phi(X)=Y$.
\end{itemize}
Lastly, in the theory of randomness with respect to a measure (computable or non-computable), a computable sequence is random with respect to some measure $\mu$ only if it is an atom of $\mu$, as shown by Reimann and Slaman \cite{2008arXiv0802.2705R}.   We extend this to the case of left-c.e.\ semi-measures:
\begin{itemize}
\item[(iv)] {\bf Computable Sequence Condition:}  If $X$ is computable and random with respect to a left-c.e.\ semi-measure $\rho$, then $\inf_n\rho(X\uh n)>0$.
\end{itemize}

\noindent With these conditions in mind, we now turn to a first candidate definition for randomness with respect to a semi-measure.

\subsection{Martin-L\"of randomness with respect to a left-c.e.\ semi-measure}

First we consider the same modification of Martin-L\"of randomness that we made in Section \ref{sec-computable-semi-measures} when defining randomness for a computable semi-measure.

\begin{Definition}
Let $\rho$ be a left-c.e.\ semi-measure. 
\begin{itemize} 
\item[(i)]  A \emph{$\rho$-Martin-L\"of test} is a sequence $(U_i)_{i\in\omega}$ of uniformly c.e.\ subsets of $\cs$ such that for each $i$,
\[
\rho(U_i)\leq 2^{-i}.
\]

\item[(ii)] $X\in\cs$ passes the $\rho$-Martin-L\"of test $(U_i)_{i\in\omega}$ if $X\notin\bigcap_{i \in \omega}\llb U_i\rrb$.

\medskip

\item[(iii)] $X\in\cs$ is \emph{$\rho$-Martin-L\"of random}, denoted $X\in\mlr_\rho$, if $X$ passes every $\rho$-Martin-L\"of test.

\end{itemize}
\end{Definition}

One interesting consequence of this definition is that the universal left-c.e.\ semi-measure $M$ is universal for Martin-L\"of randomness with respect to a left-c.e.\ semi-measure.

\begin{Proposition}\label{prop-universal-semi-measures}
Let $\mathcal{S}$ be the collection of left-c.e.\ semi-measures.  Then $\mlr_M=\bigcup_{\rho\in\mathcal{S}}\mlr_\rho$.
\end{Proposition}

\begin{proof}
Clearly $\mlr_M\subseteq\bigcup_{\rho\in\mathcal{S}}\mlr_\rho$.  For the other direction, note that for any left-c.e.\ semi-measure $\rho$, every $M$-Martin-L\"of test can be transformed into a $\rho$-Martin-L\"of test since there is some $c$ such that $\rho(\sigma)\leq c\cdot M(\sigma)$ for every $\sigma\in\str$.  Thus, if $X\notin\mlr_M$, it follows that $X\notin\mlr_\rho$ for any left-c.e.\ semi-measure $\rho$.
\end{proof}

Even though every universal left-c.e.\ semi-measure is universal in the sense of Proposition \ref{prop-universal-semi-measures}, the converse does not hold.

\begin{Proposition}\label{prop-NotUniversalButAllRandoms}
There is a non-universal left-c.e.\ semi-measure $\widetilde M$ such that
\[
\mlr_{\widetilde M} = \mlr_M =\bigcup_{\rho\in\mathcal{S}}\mlr_\rho.
\]
\end{Proposition}
\begin{proof}
First define a semi-measure $\rho$ by $\rho(\sigma)=2^{-j}M(\sigma)$, where $j$ is largest such that $1^j \preceq\sigma$.  The semi-measure $\rho$ is left-c.e., but it cannot be universal as there is no $c$ such that $c\cdot\rho(\sigma)\geq M(\sigma)$ for every $\sigma\in\str$.  

Consider a $\rho$-Martin-L\"of test $(T_i)_{i\in\omega}$.  For each $j\in\omega$ define an $M$-Martin-L\"of test $(S_i^j)_{i\in\omega}$ by $S_i^j = \{\sigma \in T_{i+j} : \sigma \succeq 1^j0\}$.  For each $\sigma\in S^j_i$, $\sigma=1^j0\tau$ for some $\tau\in\str$.  It follows that $M(\sigma)=M(1^j0\tau)=2^j\rho(1^j0\tau)=2^j\rho(\sigma)$.  Thus we have
\[
\sum_{\sigma\in S^j_i}M(\sigma)=2^j\sum_{\sigma\in S^j_i}\rho(\sigma)\leq 2^j\sum_{\sigma\in T_{i+j}}\rho(\sigma) \leq 2^j2^{-(i+j)}=2^{-i}.
\]
Clearly, every sequence containing a $0$ that is covered by $(T_i)_{i\in\omega}$ is covered by $(S_i^j)_{i\in\omega}$ for some~$j$.  Thus $\rho$ is almost the desired measure:  we have $\mlr_M \subseteq \mlr_\rho \cup \{1^\omega\}$.  Consider then the measure $\delta_{1^\omega}$ where $\delta_{1^\omega}(\sigma) = 1$ if $\sigma = 1^n$ for some $n \in \omega$ and $\delta_{1^\omega}(\sigma) = 0$ otherwise.  Let $\widetilde M = (1/2)\rho + (1/2)\delta_{1^\omega}$.  Then $\widetilde M$ is not universal, and $\mlr_{\widetilde M} \subseteq \mlr_M$ by Proposition~\ref{prop-universal-semi-measures}.  Finally, one easily checks that $\mlr_{\widetilde M} = \mlr_\rho \cup \mlr_{\delta_{1^\omega}} = \mlr_\rho \cup \{1^\omega\}$.  Hence $\mlr_M \subseteq \mlr_{\widetilde M}$ as well.
\end{proof}

Now we evaluate the adequacy of our definition in terms of the desiderata laid out in Section~\ref{subsec-desiderata}.  Clearly, this definition satisfies the condition of coherence.  Moreover, we can show that it also satisfies randomness preservation.

\begin{Theorem}\label{thm-rand-pres-1}
If $X\in\mlr$ and $\Phi$ is a Turing functional such that $X\in\dom(\Phi)$, then $\Phi(X)\in\mlr_{\lambda_\Phi}$.
\end{Theorem}

\begin{proof}
Suppose there is a $\lambda_\Phi$-Martin-L\"of test $(U_i)_{i\in\omega}$ such that $\Phi(X)\in\bigcap_{i\in\omega}\llb U_i\rrb$.  Then $(\llb\Phi^{-1}(U_i)\rrb)_{i\in\omega}$ is a uniform sequence of $\Sigma^0_1$ subsets of $\cs$, and
\[
\lambda(\llb\Phi^{-1}(U_i)\rrb)=\lambda_\Phi(U_i)\leq 2^{-i}
\]
for every $i$, so $(\llb\Phi^{-1}(U_i)\rrb)_{i\in\omega}$ is a Martin-L\"of test containing $X$.
\end{proof}

\begin{Remark}
Despite satisfying these two conditions, in general $\rho$-Martin-L\"of randomness fails to satisfy the No Randomness Ex Nihilo principle and the computable sequence condition.  First for the counterexample to the No Randomness Ex Nihilo principle, let $\rho$ be the semi-measure constructed in the proof of Theorem~\ref{thm-shen-counterexample}.  There we constructed functionals $\Phi$ and $\Psi$ inducing $\rho$ such that $\dom(\Phi)=\{\Omega\}$, $\dom(\Psi)=\{0^\omega\}$, and $\Phi(\Omega)=0^\omega=\Psi(0^\omega)$.  By Theorem~\ref{thm-rand-pres-1}, $0^\omega\in\mlr_\rho$.  However, $\Psi$ induces $\rho$ and yet maps no Martin-L\"of random sequence to $0^\omega$.  The same example provides a counterexample to the computable sequence condition:  $0^\omega$ is $\rho$-Martin-L\"of random, but $\inf_n\rho(0^n)=0$.
%, as discussed in Remark \ref{rem-shen-counterexample-nonatom}.
\end{Remark}

We can also construct a left-c.e.\ semi-measure $\rho$ that fails to satisfy the computable sequence condition in the strongest possible way:  $\rho$ has no atoms and yet $\mlr_\rho=\cs$.

\begin{Theorem}\label{thm-EverythingMLR}
There is a non-atomic left-c.e.\ semi-measure $\rho$ such that every $X \in 2^\omega$ is Martin-L\"{o}f random for $\rho$.
\end{Theorem}

\begin{proof}
Let $(E^e_n)_{\la e,n\ra \in \omega}$ be an effective list of all uniformly c.e.\ sequences of subsets of $2^{<\omega}$.  We satisfy the requirements
\begin{align*}
\MP{R}_e \colon \bigcap_{n \in \omega}\llb E^e_n\rrb \neq \emptyset \imp (\exists n \in \omega)(\rho(E^e_n) > 2^{-n}).
\end{align*}
Satisfying all of these requirements ensures that if $(E_n)_{n\in\omega}$ defines a $\rho$-Martin-L\"of test, then $\bigcap_{n \in \omega}\llb E^e_n\rrb = \emptyset$.  Therefore every $X \in 2^\omega$ is $\rho$-Martin-L\"of random.

For each~$e$, we build a left-c.e.\ semi-measure $\rho_e$ (were we relax the requirement $\rho_e(\emptystr) = 1$ to $\rho_e(\emptystr) \leq 1$) as follows.
\begin{itemize}
\item Start with $\rho_e(\sigma)=0$ for all~$\sigma$.
\item If at some stage some $\tau$ enters $E^e_{e+2}$, set $\rho_e(\tau')=2^{-e-1}$ for all prefixes of~$\tau$ (including $\tau$ itself) and finish the construction.
\end{itemize}

Clearly, $\rho_e$ is a left-c.e.\ semi-measure such that $\rho_e(E^e_{e+2}) > 2^{-e-2}$ if $E^e_{e+2} \not= \emptyset$, and $\rho_e(\emptystr) \leq 2^{-e-1}$. Thus, define $\rho$ by $\rho(\emptystr) = 1$ and $\rho(\sigma) = \sum_{e \in \omega} \rho_e(\sigma)$ for all $\sigma$ with $|\sigma|>0$.  Then $\rho$ is a left-c.e.\ semi-measure such that  $\rho(E^e_{e+2}) > 2^{-e-2}$ if $E^e_{e+2} \not= \emptyset$.
\end{proof}

Note that Proposition \ref{prop-universal-semi-measures} and Theorem \ref{thm-EverythingMLR} together imply the following.

\begin{Corollary}
$\mlr_M=\cs$. 
\end{Corollary}

In light of the fact that $\rho$-Martin-L\"of randomness does not always satisfy the desiderata from Section \ref{subsec-desiderata}, we consider other definitions of randomness for a left-c.e.\ semi-measure.

\subsection{Weak 2-randomness with respect to a left-c.e.\ semi-measure}\label{subsec-w2r}

We can obtain the definition of weak 2-randomness for a left-c.e.\ semi-measure by modifying the notion of a generalized Martin-L\"of test.

\begin{Definition}
Let $\rho$ be a left-c.e.\ semi-measure. 
\begin{itemize} 
\item[(i)]  A \emph{generalized $\rho$-Martin-L\"of test} is a sequence $(U_i)_{i\in\omega}$ of uniformly c.e.\ subsets of $\cs$ such that for each $i$,
\[
\lim_{i\rightarrow\infty}\rho(U_i)=0.
\]

\item[(ii)] $X\in\cs$ passes the generalized $\rho$-Martin-L\"of test $(U_i)_{i\in\omega}$ if $X\notin\bigcap_{i \in \omega}\llb U_i\rrb$.

\medskip

\item[(iii)] $X\in\cs$ is \emph{$\rho$-weakly 2-random}, denoted $X\in\wtwor_\rho$, if $X$ passes every generalized $\rho$-Martin-L\"of test.
\end{itemize}
\end{Definition}

Weak 2-randomness for a left-c.e.\ semi-measure is more well-behaved than the previous definition considered in this section, as it satisfies both randomness preservation and the computable sequence condition.

\begin{Theorem}
Let $\rho$ be a left-c.e.\ semi-measure, and let $\Phi$ be a Turing functional that induces~$\rho$.  Then for every $X \in\wtwor\cap\dom(\Phi)$, $\Phi(X)\in\wtwor_\rho$.
\end{Theorem}

\begin{proof}
The proof is nearly identical to the proof of Theorem \ref{thm-rand-pres-1}.
\end{proof}

\begin{Proposition}
Let $\rho$ be a left-c.e.\ semi-measure.  Suppose that $X$ is computable and that $X\in\wtwor_\rho$.  Then $\inf_n\rho(X\uh n)>0$.
\end{Proposition}

\begin{proof}
Suppose that $X$ is computable and $\inf_n\rho(X\uh n)=0$.  Then setting $U_i=\{X\uh i\} $ for each $i\in\omega$ yields a generalized $\rho$-Martin-L\"of test capturing $X$.
\end{proof}

Clearly, $\wtwor_\rho\subseteq\mlr_\rho$, but for some semi-measures $\rho$ (such as any $\rho$ such that $\rho$-Martin-L\"of randomness violates the computable sequence condition), the inclusion is strict.  We should note further that the universal left-c.e.\ semi-measure $M$ is universal for weak 2-randomness, as is the non-universal $\widetilde M$ from Proposition \ref{prop-NotUniversalButAllRandoms}:
\[
\wtwor_{\widetilde M} = \wtwor_M =\bigcup_{\rho\in\mathcal{S}}\wtwor_\rho.
\]

We have seen that weak 2-randomness for a semi-measure satisfies coherence, randomness preservation, and the computable sequence condition, but we currently do not know whether it satisfies the No Randomness Ex Nihilo principle.  We will return to this question at the end of Section \ref{section-trimming}.

We now turn to another general approach to defining randomness with respect to a semi-measure, an approach that is found implicitly in \cite{LevVyu77} and \cite{Vyu82}.

\section{Trimming a semi-measure back to a measure}\label{section-trimming}

We can also define randomness with respect to a semi-measure by trimming back our semi-measure to a measure and then considering the sequences that are random with respect to the resulting measure.  

\subsection{Definition of a derived measure and examples}

To better understand this approach, it is helpful to think of a semi-measure as a network flow through the full binary tree $\str$ seen as a directed graph.  We initially give the node at the root of the tree flow equal to $1$, which implies that $\rho(\emptystr)=1$. Some amount of this flow at each node $\sigma$ is passed along to the node corresponding to $\sigma0$, some is passed along to the node corresponding to $\sigma1$, and, potentially, some of the flow is lost, yielding the condition that $\rho(\sigma)\geq\rho(\sigma0)+\rho(\sigma1)$.  

We obtain a measure $\overline\rho$ from $\rho$ if we ignore all of the flow that is lost and just consider the behavior of the flow that never leaves the network. We will refer to $\overline\rho$ as the measure derived from $\rho$.  This can be formalized as follows.

\begin{Definition}\cite{LevVyu77}
Let $\rho$ be a semi-measure.  
\[
\overline\rho(\sigma):=\inf_{n\geq |\sigma|}\sum_{\tau\succeq\sigma\;\&\;|\tau|=n}\rho(\tau) = \lim_{n\rightarrow \infty} \sum_{\tau\succeq\sigma\;\&\;|\tau|=n}\rho(\tau).
\]
\end{Definition}
\noindent The fact that one can use either $\inf$ or $\lim$ in the expression is due to the fact that the term $\sum_{\tau\succeq\sigma\;\&\;|\tau|=n}\rho(\tau)$ is non-increasing in~$n$ by the semi-measure inequality $\rho(\tau) \geq \rho(\tau0)+\rho(\tau1)$.

The following are two simple examples illustrating the different behaviors of $\rho$ and $\overline\rho$.
\begin{Example}
Let $\rho(\sigma)=4^{-|\sigma|}$ for every $\sigma\in\str$.  Then for each $\sigma\in\str$ and each $n\in\omega$,
\[
\sum_{\tau\succeq\sigma\;\&\;|\tau|=n}\rho(\tau)=2^{n-|\sigma|}4^{-n}=2^{-n}2^{-|\sigma|}.
\]
Thus $\overline\rho(\sigma)=0$ for every $\sigma\in\str$.
\end{Example}
\begin{Example}
Let $\rho$ be a semi-measure such that $\rho(\sigma)=\dfrac{1}{2}\lambda(\sigma)+\dfrac{1}{2^{2|\sigma|+1}}$. Then for each $\sigma\in\str$ and each $n\in\omega$,
\[
\begin{split}
\sum_{\tau\succeq\sigma\;\&\;|\tau|=n}\rho(\tau)&=2^{n-|\sigma|}\Bigl(\dfrac{1}{2}\lambda(\tau)+\dfrac{1}{2^{2|\tau|+1}}\Bigr)\\
&=2^{n-|\sigma|}\Bigl(\dfrac{1}{2}2^{-n}+\dfrac{1}{2^{2n+1}}\Bigr)\\
&=\dfrac{1}{2}2^{-|\sigma|}+\dfrac{2^{n-|\sigma|}}{2^{2n+1}}\\
&=\dfrac{1}{2}\lambda(\sigma)+2^{-(n+1)}\lambda(\sigma).
\end{split}
\]
Thus $\overline\rho(\sigma)=\dfrac{1}{2}\lambda(\sigma)$ for every $\sigma\in\str$.
\end{Example}
\noindent This latter example yields what we will refer to as a \emph{Lebesgue-like} semi-measure.
\begin{Definition}
A semi-measure $\rho$ is \emph{Lebesgue-like} if there is some $\alpha\in(0,1]$ such that
\[
\overline\rho=\alpha\cdot\lambda.
\]
\end{Definition}

Let us now show that $\overline{\rho}$ is indeed a measure which enjoys some nice properties, both from the analytic viewpoint and in connection with Turing functionals. The following proposition is probably folklore; an explicit reference is hard to find in the literature.

\begin{Proposition}\label{prop-rhobar-properties}
Let $\rho$ be a semi-measure and $\overline{\rho}$ be defined as above. Then $\overline{\rho}$ is the largest measure $\mu$ such that $\mu\leq \rho$.  In particular, if $\rho$ is a measure, then $\overline\rho=\rho$. Moreover, if
\[
\rho(\sigma)=\lambda(\{X:\Phi^X\succeq\sigma\}),
\]
then
\[
\overline\rho(\sigma)=\lambda(\{X \in \dom(\Phi): \Phi^X\succeq\sigma\}).
\]
(Thus, trimming $\rho$ back to $\overline\rho$ amounts to restricting the Turing functional $\Phi$ that induces $\rho$ to those inputs on which $\Phi$ is total.)
\end{Proposition}

\begin{proof}
The fact that $\overline{\rho}$ is a measure is clear from the definition $\overline\rho(\sigma) = \lim_{n \rightarrow \infty}\sum_{\tau\succeq\sigma\;\&\;|\tau|=n}\rho(\tau)$, since we then have
\[
\overline\rho(\sigma0)+\overline\rho(\sigma1) = \lim_{n \rightarrow \infty}\sum_{\tau\succeq\sigma0\;\&\;|\tau|=n}\rho(\tau) + \lim_{n\rightarrow\infty}\sum_{\tau\succeq\sigma1\;\&\;|\tau|=n}\rho(\tau) =  \lim_{n \rightarrow \infty}\sum_{\tau\succeq\sigma\;\&\;|\tau|=n}\rho(\tau) = \overline\rho(\sigma).
\]
Now, if $\mu$ is a measure such that $\mu(\sigma) \leq \rho(\sigma)$ for all~$\sigma$, then for any given $\sigma$:
\[
\mu(\sigma) = \inf_{n\geq |\sigma|}\sum_{\tau\succeq\sigma\;\&\;|\tau|=n}\mu(\tau) \leq \inf_{n\geq |\sigma|}\sum_{\tau\succeq\sigma\;\&\;|\tau|=n}\rho(\tau) = \overline{\rho}(\sigma)
\]
(for the first equality, we used the measure property $\mu(\tau)=\mu(\tau0)+\mu(\tau1)$).  

Suppose now that $\rho$ is induced by some Turing functional~$\Phi$, i.e., $\rho(\sigma)=\lambda(\{X:\Phi^X\succeq\sigma\})$ for all~$\sigma$. Set $\mu(\sigma)= \lambda(\{X \in \dom(\Phi) :\Phi^X\succeq\sigma\})$. 

Let $\mathcal{D}_n$ be the set of $X$ such that $\Phi^X$ is of length at least~$n$. The sets $\mathcal{D}_n$ are 

non-increasing in~$n$. Moreover, $\dom(\Phi) = \bigcap_{n \in \omega} \mathcal{D}_n$. Therefore, for all $\sigma$:
\[
\mu(\sigma) = \lim_{n \rightarrow \infty} \lambda (\{X \in \mathcal{D}_n : \Phi^X \succeq \sigma\}).
\]
By definition, for $n \geq |\sigma|$, 
\[
\lambda (\{X \in \mathcal{D}_n : \Phi^X \succeq \sigma\}) = \sum_{\tau\succeq\sigma\;\&\;|\tau|=n}\rho(\tau).
\]
Putting the two together, 
\[
\mu(\sigma) = \lim_{n \rightarrow \infty}   \sum_{\tau\succeq\sigma\;\&\;|\tau|=n}\rho(\tau) = \overline{\rho}(\sigma)
 \]
 as wanted.
\end{proof}

%If $M$ is the universal left-c.e.\ semi-measure induced by the universal Turing functional, it follows from Proposition \ref{prop-rhobar-properties} that for every left-c.e.\ semi-measure $\rho$, 
%there exists $c\in\omega$ such that $\overline\rho\leq c\cdot \overline M$.  Then following the proof of Proposition \ref{prop-universal-semi-measures} yields:

\subsection{The complexity of $\overline\rho$}

We now show that for a given left-c.e.\ semi-measure $\rho$, $\overline\rho$ can encode a lot of information.  More precisely, for any $\emptyset'$-right-c.e.\ real $\alpha$ (i.e.\ $\alpha$ is the limit of a $\emptyset'$-computable non-increasing sequence of rationals), we code $\alpha$ into the values of $\overline\rho$ for some left-c.e.\ semi-measure $\rho$.  Further, we can even make $\overline\rho$ Lebesgue-like, as shown by the next theorem (the equivalence of $(1)$, $(2)$ and $(3)$ is well-known but it is hard to find a reference for this result, so we include the proof for completeness).

\begin{Theorem}\label{thm-complexity-of-rhobar}
The following are equivalent for $\alpha\in[0,1]$.
	\begin{enumerate}
	\item $\alpha$ is $\jump$-right c.e.
	\item $\alpha = \limsup_n q_n$ for a computable sequence of rationals $(q_n)_{n\in\omega}$.
	\item $\alpha = \inf r_n$ where $(r_n)_{n\in\omega}$ is a uniform sequence of left-c.e.\ reals.
	\item There is a left-c.e.\ semi-measure $\rho$ such that $\br{\rho}=\alpha\cdot\lambda$.
	\end{enumerate}
\end{Theorem}

\begin{proof}
\noindent (1) $\Rightarrow$ (2): Let $\alpha \in [0,1]$ be $\jump$-right c.e, and assume that $\alpha$ is irrational because the implication is clear for rational $\alpha$.  Thus there is a $\jump$-computable function $g$ such that $(g(i))_{i \in \omega}$ is a strictly decreasing sequence of rationals in $[0,1]$ converging to $\alpha$.  By the limit lemma, there is a total  computable function~$f$ that outputs rationals in $[0,1]$ and is such that $(\forall i \in \omega)(g(i) = \lim_s f(i,s))$.

We define our sequence of rationals $(q_n)_{n\in\omega}$ as follows.  Let $(i_s)_{s \in \omega}$ be an effective sequence of natural numbers in which every number appears infinitely often.  At stage $s$, enumerate $f(i_s,s)$ as the next rational in the sequence if it has not yet been enumerated and $(\forall k < i_s)(f(i_s,s) < f(k,s))$.

We show that, for every $i \in \omega$,
\begin{itemize}
\item[(i)] $(\exists n_0 \in \omega)(q_{n_0} = g(i))$, and 
\item[(ii)] $(\exists n_1 \in \omega)(\forall n > n_1)(q_n < g(i))$.
\end{itemize}
For (i), given $i$, let $s$ be such that $(\forall k \leq i)(g(k) = f(k,s))$ and $i_s = i$.  Then at stage $s$ we have $(\forall k < i_s)(f(i_s,s) = g(i_s) < g(k) = f(k,s))$, so at this stage $f(i_s,s) = g(i)$ will be enumerated if it has not been enumerated already.  For (ii), given $i$, let $s_0$ be such that $(\forall k \leq i)(\forall s \geq s_0)(g(k) = f(k,s))$ and such that (by (i)) every $g(k)$ for $k \leq i$ has been enumerated by stage $s_0$.  Consider an $f(i_s,s)$ that is enumerated at some stage $s > s_0$.  
It is impossible that $i_s \leq i$ because in this case at stage $s$ we would have $f(i_s,s) = g(i_s)$, and by assumption this number was already enumerated.
%
%If $i_s \leq i$, then $f(i_s,s)$ was in fact not enumerated at stage $s$ because $f(i_s,s) = g(i_s)$ which by assumption was already enumerated.  
%
Thus $i_s > i$, and to be enumerated at stage~$s$, $f(i_s,s)$ must satisfy $f(i_s,s) < f(i,s) = g(i)$ as desired.

The conclusion $\alpha = \limsup_n q_n$ now follows from (i) and (ii).  By (i), every tail of the sequence $(q_n)_{n \in \omega}$ contains an element of the form $g(i)$ for some $i$, hence since $\alpha < g(i)$ we have $\alpha \leq \limsup_n q_n$.  By (ii), $(\forall i \in \omega)(\limsup_n q_n \leq g(i))$, hence $\limsup_n q_n \leq \alpha$.\\

\noindent (2) $\Rightarrow$ (3): Suppose that $\alpha = \limsup_n q_n$ for a computable sequence of rationals $(q_n)_{n\in\omega}$.  Let $r_n := \sup(q_i)_{i\geq n}$. Clearly each $r_n$ is left-c.e.\ and $\inf_n r_n=\limsup q_n=\alpha$.\\

\noindent (3) $\Rightarrow$ (4): Since each $r_i$ is left-c.e., let $r_{i,s}$ be the $s^{\mathrm{th}}$ rational in the approximation of $r_i$.  To define $\rho$, we let $\displaystyle\rho_s(\sigma)=2^{-|\sigma|}\min_{i\leq|\sigma|} r_{i,s}$.  Then $\rho(\sigma)=2^{-|\sigma|}\min_{i\leq|\sigma|}r_i $.  It is routine to verify that~$\rho$ is a semi-measure.  Now observe that
\[
\begin{split}
\overline\rho(\sigma)=\inf_n\sum_{\tau\succeq\sigma\;\&\;|\tau|=n}\rho(\tau)&=\inf_n\sum_{\tau\succeq\sigma\;\&\;|\tau|=n} 2^{-|\tau|}\min_{i\leq|\tau|}r_i\\
&=\inf_n\; 2^{n-|\sigma|}2^{-n}\min_{i\leq n}r_i=2^{-|\sigma|}\inf_n\min_{i\leq n}r_i=\alpha\cdot2^{-|\sigma|}.
\end{split}
\]
 \\

\noindent (4) $\Rightarrow$ (1): $\jump$ computes $\sum_{x\colon |x|=n} \rho(x)$ uniformly in $n$. Therefore $\br{\rho}(\emptystr)=\inf_n \sum_{x\colon |x|=n} \rho(x)$ is $\jump$-right-c.e.
\end{proof}

The following corollary tells us that $\overline\rho$ can be as complicated as possible.

\begin{Corollary}\label{cor-lebesgue-like}
There is a left-c.e.\ semi-measure $\rho$ such that $\overline\rho=\alpha\cdot\lambda$ and $\alpha\equiv_T\emptyset''$.  In particular, every representation of $\overline\rho$ computes $\emptyset''$.
\end{Corollary}

\begin{proof}
Recall that $\mathsf{Tot}=\{e:\Phi_e\;\text{is total}\}$.  Let $\alpha=\sum_{e \in \mathsf{Tot}}2^{-e}$, which is $\jump$-right-c.e., and apply Theorem \ref{thm-complexity-of-rhobar}.
\end{proof}

%\chris{New material (please modify as you see fit):} 
%
%One consequence of this corollary is that $\overline{M}$ is also as complicated as possible, where $M$ is the universal semi-measure induced by the universal Turing functional defined at the end of Section \ref{sec-prelims}. 
%
%\begin{Corollary}\label{cor-complexity-Mbar}
%Every representation of $\overline M$ computes $\emptyset''$.
%\end{Corollary}
%
%\begin{proof}
%Let $\rho$ be the left-c.e.\ semi-measure from Corollary \ref{cor-lebesgue-like}.  Then if $\rho$ is induced by the $e^{\mathrm{th}}$ Turing functional $\Phi_e$, it follows that
%\[
%\overline M(1^e0)=2^{-(e+1)}\overline\rho(\emptystr)=2^{-(e+1)}\alpha,
%\]
% where $\alpha\equiv_T\emptyset''$.  Since $\overline M(1^e0)$ is computable from every representation of $\overline M$, the result follows.
%\end{proof}
%\chris{end of new material}

Despite the fact that for a given left-c.e\ semi-measure $\rho$, $\overline\rho$ can encode a lot of information, we cannot obtain every $\jump$-computable measure in this way, as the following result shows.  The witnessing measure $\mu$ we construct even has a low representation in the sense described at the beginning of Section~\ref{subsec-non-computable-measures} because the (in this case rational-valued) function $\sigma \mapsto \mu(\sigma)$ is low and clearly computes a representation of $\mu$.

\begin{Proposition}
There is a measure $\mu$ such that $\mu(\sigma)$ is a positive rational for all strings~ $\sigma$, the function $\sigma \mapsto \mu(\sigma)$ is low, and $\mu \neq \alpha\cdot\overline\rho$ for every left-c.e.\ real $\alpha$ and every left-c.e.\ semi-measure~$\rho$ (in particular, $\mu \neq \overline\rho$ for any left-c.e.\ semi-measure~$\rho$). 
\end{Proposition}

\begin{proof}
Let $\Qb^{>0}$ denote the set of positive rationals.  For each $n \in \omega$, let $2^{\leq n}$ denote the set of strings of length at most $n$, and let $2^{< n}$ denote the set of strings of length less than~$n$.  Define a \emph{partial measure} to be a function of the form $m \colon 2^{\leq n} \imp \Qb^{>0}$ for some $n \in \omega$ such that $m(\emptystr) = 1$ and $(\forall \sigma \in 2^{< n})(m(\sigma) = m(\sigma0) + m(\sigma1))$.  The partial measures form a partial order $\Pb$ when ordered by extension:  $m_0 \sqsubseteq m_1$ if $\dom(m_0) \supseteq \dom(m_1)$ and $(\forall \sigma \in \dom(m_1))(m_0(\sigma) = m_1(\sigma))$.  Similarly, if $m$ is a partial measure and $\mu$ is a measure, we write $\mu \sqsubseteq m$ if $(\forall \sigma \in \dom(m))(\mu(\sigma) = m(\sigma))$.  

To ensure $\mu \neq \alpha\cdot\overline\rho$, it suffices to ensure that there is a $\sigma \in 2^{<\omega}$ such that $\mu(\sigma) > \alpha\cdot\rho(\sigma)$ because then $\mu(\sigma) > \alpha\cdot\rho(\sigma) \geq \alpha\cdot\overline\rho(\sigma)$.  To this end, let $(\alpha_e)_{e \in \omega}$ be an effective list of all left-c.e.\ reals, and let $(\rho_e)_{e \in \omega}$ be an effective list of all left-c.e.\ semi-measures.

We satisfy the following list of requirements for all $e,i \in \omega$:
\begin{align*}
\MP{R}_{\la e,i \ra} \colon& (\exists \sigma \in 2^{<\omega})(\mu(\sigma) > \alpha_e\cdot\rho_i(\sigma))\\
\MP{L}_e \colon& (\exists m \sqsupseteq \mu)(\Phi_e^m(e)\da \;\vee\; (\forall m' \sqsubseteq m)(\Phi_e^{m'}(e)\ua)).
\end{align*}

To each requirement we associate the subset of $\Pb$ consisting of the partial measures that satisfy the requirement:
\begin{align*}
R_{\la e,i \ra} &= \{m \in \Pb : (\exists \sigma \in \dom(m))(m(\sigma) > \alpha_e \cdot \rho_i(\sigma))\}\\
L_e &= \{m \in \Pb : \Phi_e^m(e)\da \;\vee\; (\forall m' \sqsubseteq m)(\Phi_e^{m'}(e)\ua)\}.
\end{align*}

\begin{Claim}
For every $e,i \in \omega$, $R_{\la e,i \ra}$ is a dense subset of $\Pb$.
\end{Claim}
\begin{proof}
Let $m \colon 2^{\leq n} \imp \Qb^{>0}$ be a given member of $\Pb$, and let $q = m(0^n)$.  The fact that $\rho_i$ is a semi-measure implies that, for all $k \geq n$, $\sum\{\rho_i(\sigma) : \sigma \succeq 0^n \;\wedge\; |\sigma| = k\} \leq \rho_i(0^n)$.  Therefore $\inf\{\rho_i(\sigma) : \sigma \succeq 0^n\} = 0$, so there is a $\sigma \succ 0^n$ such that $\alpha_e \cdot \rho_i(\sigma) \leq q/2$.  We may extend $m$ to a partial measure $m'$ that satisfies $m'(\sigma) = 3q/4$ and $m'(\tau) = q/4(2^{|\sigma|-n}-1)$ for all $\tau \succeq 0^n$ with $|\tau| = |\sigma|$ and $\tau \neq \sigma$.  Then $m' \in R_{\la e,i \ra}$ because $m'(\sigma) = 3q/4 > q/2 \geq \alpha_e \cdot \rho_i(\sigma)$.   
\end{proof}

\begin{Claim}
For every $e \in \omega$, $L_e$ is a dense subset of $\Pb$.
\end{Claim}
\begin{proof}
Let $m$ be given.  If there is an $m' \sqsubseteq m$ such that $\Phi_e^{m'}(e)\da$, then $m' \in L_e$.  If not, then $m \in L_e$.
\end{proof}

The sets $R_{\la e,i \ra}$ and $L_e$ are dense in $\Pb$ and uniformly c.e.\ in $\jump$ ($L_e$ is even $\jump$-computable), so $\jump$ can compute a measure $\mu$ such that $\mu(\sigma)$ is a positive rational for each string $\sigma$ and such that $\mu$ meets all of the requirements.  That is, $(\forall e,i \in \omega)(\exists m \sqsupseteq \mu)(m \in R_{\la e, i\ra})$ and $(\forall e \in \omega)(\exists m \sqsupseteq \mu)(m \in L_e)$.  Therefore $\sigma \mapsto \mu(\sigma)$ is low, and $\mu \neq \alpha \cdot \overline\rho$ for every left-c.e.\ real $\alpha$ and every left-c.e.\ semi-measure $\rho$.\end{proof}

It is well-known that for a measure $\mu$, the atoms of $\mu$ are computable from any representation of $\mu$ (which can be shown by generalizing the proof of Lemma \ref{lem-kautz}).  Thus, given the computational power of $\overline\rho$, one might expect that the atoms of~$\overline\rho$ for some left-c.e.\ semi-measure~$\rho$ will include some non-computable sequences.  But this does not hold.

\begin{Proposition}\label{prop-no-new-atoms}
A set $X \in 2^\omega$ is computable if and only if there exists a left-c.e.\ semi-measure~ $\rho$ such that $X$ is an atom of $\overline\rho$.
\end{Proposition}
\begin{proof} 
The left to right direction is trivial:  for a given computable $X\in\cs$, we define a left-c.e.\ semi-measure $\rho$ by setting $\rho(X\uh n)=1$ for all $n$. For the other direction let $\rho$ be a left-c.e.\ semi-measure and assume that $X$ is an atom of $\br \rho$. Write $\alpha=\lim_n \br\rho(X\uh n)$ and choose $q\in\Qb$ such that $\frac{1}{2}\alpha < q< \alpha$.
Then there exists a large enough $N$ such that $\rho(X\uh N)$ is strictly smaller than $2q$. To decide all further bits of $X$, say $X(n)$ for $n\geq N$ we proceed inductively as follows.  Wait until one of $\rho(X\uh n^\frown 0)$ and $\rho(X\uh n^\frown 1)$ attains or exceeds $q$, and output the according bit.  This bit is the correct value of $X(n)$, since $\rho(X\uh n^\frown X(n))$ must eventually attain or exceed $q$ while $\rho(X\uh n^\frown (1-X(n)))$ cannot attain $q$, as otherwise their sum would be at least $2q$ and would therefore exceed $\rho(X\uh N)$, contradicting our choice of~$N$.
\end{proof}

Although the measure derived from a left-c.e.\ semi-measure cannot have a non-computable atom, one interesting difference between these derived measures and computable measures is that whereas there is no computable measure $\mu$ such that every computable sequence is a $\mu$-atom (because for each computable measure $\mu$ one can effectively find a sequence~$X$ such that $\lim_{n\rightarrow\infty}\mu(X\uh n)=0$), \emph{every} computable sequence is an atom of $\overline M$, because $\overline M$ dominates every computable measure up to a positive multiplicative constant.

\subsection{Notions of randomness with respect to $\overline\rho$}

We now apply the definitions of Martin-L\"of randomness with respect to non-computable measures, introduced in Section \ref{subsec-non-computable-measures}, to the measure derived from a semi-measure and compare the resulting definitions to the definitions studied in Section \ref{sec-direct-modifications}.

As noted in Section \ref{subsec-non-computable-measures}, there are two general approaches to defining a randomness test $(\MP{U}_i)_{i\in\omega}$  with respect to a non-computable measure $\mu$:  either allow $(\MP{U}_i)_{i\in\omega}$ to have access to a representation of $\mu$ as an oracle and require $\mu(\MP{U}_i)\leq 2^{-i}$ for every $i$, or simply require the latter condition without using a representation of $\mu$ as an oracle.  

Taking the former approach yields the following example:
\begin{Proposition}\label{prop-lebesgue-like}
Let $\rho$ be the semi-measure from Corollary \ref{cor-lebesgue-like}, so that $\overline\rho=\alpha\cdot\lambda$ for some $\alpha\equiv_T\emptyset''$.   Then $\overline\rho$-Martin-L\"of randomness is 3-randomness.
\end{Proposition}

\begin{proof}
Let $j\in\omega$ satisfy $2^{-(j+1)}<\alpha<2^{-j}$, which implies that $2^j <\frac{1}{\alpha}<2^{(j+1)}$.  First we show that $\mlr^{\emptyset''}\subseteq\mlr_{\overline\rho}$.  Since
 $\emptyset''$ computes a representation of $\overline\rho$, we have $\mlr^{\emptyset''}_{\overline\rho}\subseteq\mlr_{\overline\rho}$.  Now for any $\emptyset''$-Martin-L\"of test $(\MP{U}_i)_{i\in\omega}$ (with respect to $\overline\rho$), we have $\alpha\cdot\lambda(\MP U_i)\leq 2^{-i}$, which implies that $\lambda(\MP U_i)\leq 2^{j+1-i}$.  Thus $(\MP U_i)_{i\geq j+1}$ is a $\emptyset''$-Martin-L\"of test (with respect to $\lambda$) that covers $(\MP{U}_i)_{i\in\omega}$.  Thus $\mlr^{\emptyset''}\subseteq\mlr_{\overline\rho}$.

To show that $\mlr_{\overline\rho}\subseteq\mlr^{\emptyset''}$, let  
$(\MP{U}_i)_{i\in\omega}$ be a $\emptyset''$-Martin-L\"of test with respect to $\lambda$.  Then since
\[
\alpha\cdot\lambda(\MP U_i)\leq2^{-j}\lambda(\MP U_i)\leq 2^{-(i+j)},
\]
it follows that $(\MP{U}_i)_{i\in\omega}$ is a $\emptyset''$-Martin-L\"of test with respect to $\overline\rho$.  But since every representation of $\overline\rho$ computes $\emptyset''$, it follows that for any such representation $R$, $(\MP{U}_i)_{i\in\omega}$ is an $R$-Martin-L\"of test with respect to $\overline\rho$.  Thus $\mlr_{\overline\rho}^R\subseteq\mlr^{\emptyset''}$ for all representations $R$ of $\overline\rho$, and hence $\mlr_{\overline\rho}\subseteq\mlr^{\emptyset''}$.
\end{proof}

This example shows a defect of using $\overline\rho$ to define randomness with respect to $\rho$.  As $\overline\rho$ is a multiple of the Lebesgue measure, we would expect that $\overline\rho$-Martin-L\"of randomness is just Martin-L\"of randomness with respect to the Lebesgue measure.  But the $\alpha$ encodes information that can be used to derandomize any sequence that is not 3-random.  However, the blind approach to $\overline\rho$-randomness avoids this problem.

\begin{Proposition}
Let $\rho$ be the semi-measure from Corollary \ref{cor-lebesgue-like}.  Then blind $\overline\rho$-Martin-L\"of randomness is the same as Martin-L\"of randomness.
\end{Proposition}

\begin{proof}
By an argument similar to the one in the proof of Proposition \ref{prop-lebesgue-like}, every Martin-L\"of test is covered by a blind $\overline\rho$-Martin-L\"of test, and vice versa.
\end{proof}

As a consequence of these two examples, we have the following:

\begin{Corollary}
There is a left-c.e.\ semi-measure $\rho$ such that $\mlr_{\overline\rho}\subsetneq\bmlr_{\overline\rho}$.
\end{Corollary}

We also have the following.

\begin{Proposition}
There is a left-c.e.\ semi-measure $\rho$ such that $\bmlr_{\overline\rho}\subsetneq\mlr_\rho$.
\end{Proposition}

\begin{proof}
Let $\rho$ be the left-c.e.\ semi-measure from the proof of Theorem \ref{thm-shen-counterexample}, so that $\mlr_\rho=\{0^\omega\}$.  Since $\rho$ is induced by a functional $\Phi$ such that $\lambda(\dom(\Phi))=0$, by the characterization of $\overline\rho$ given in Proposition \ref{prop-rhobar-properties}, $\overline\rho(\cs)=\overline\rho(\llb\emptystr\rrb)=\lambda\{X:X\in\dom(\Phi)\}=0$.  Thus ${\mlr_{\overline\rho}=\emptyset}$. 
\end{proof}

The above proof also shows that $\bmlr_{\overline\rho}$ does not satisfy randomness preservation, since~$\Phi$ induces $\rho$ (and hence $\overline\rho$), but $\Phi(\mlr)=\{0^\omega\}\neq\bmlr_{\overline\rho}$.  Thus, blind Martin-L\"of randomness for $\overline\rho$ does not provide an adequate definition of randomness for $\rho$ according to the desiderata laid out in Section \ref{subsec-desiderata}.

Blind weak 2-randomness with respect to $\overline\rho$  fares much better than $\overline\rho$-Martin-L\"of randomness and blind Martin-L\"of randomness with respect to  $\overline\rho$.
As we now show, blind weak 2-randomness for $\overline\rho$ is equivalent to weak 2-randomness for $\rho$, and hence satisfies randomness preservation.  First, we need a lemma generalizing the definition of $\overline\rho(\sigma)$.

\begin{Lemma}\label{lem-LimitHelper}
Let $\rho$ be a left-c.e.\ semi-measure.  Let $E \subseteq 2^{<\omega}$ be prefix-free.  For each $m \in \omega$, let $E^m = \{\sigma \in 2^{<\omega} : (\exists \tau \in E)(\tau \preceq \sigma \andd |\sigma| = |\tau|+m)\}$.  Then $\overline\rho(\llb E\rrb) = \lim_{m \imp \infty}\rho(E^m)$.
\end{Lemma}

\begin{proof}
For all $m \in \omega$, $\rho(E^{m+1}) \leq \rho(E^m)$.  Thus it suffices to show that for every $k\in\omega$ there is some $m\in\omega$ such that $\rho(E^m) \leq \overline\rho(\llb E\rrb) + 1/k$.

Recall that, for all $\tau \in 2^{<\omega}$, $\overline\rho(\tau) = \inf_m\sum\{\rho(\sigma) : \sigma \succeq \tau \andd |\sigma| = |\tau|+m\}$.  Thus if $E$ is finite, then for all $k \in \omega$ there is an $m \in \omega$ such that $\rho(E^m) \leq \overline\rho(\llb E\rrb)+1/k$.  Suppose instead that $E$ is infinite, and let $k \in \omega$.  The fact that $E$ is prefix-free implies that $\rho(E)$ is finite.  Thus there is an $\ell \in \omega$ such that $\sum\{\rho(\sigma) : \sigma \in E \andd |\sigma| > \ell\} < 1/2k$.  Now let $E_0 = \{\tau \in E : |\tau| \leq \ell\}$, let $E_1 = \{\tau \in E : |\tau| > \ell\}$, and let $m$ be such that $\rho(E_0^m) \leq \overline\rho(\llb E_0\rrb) + 1/2k$.  Then
\[\rho(E^m) = \rho(E_0^m) + \rho(E_1^m) \leq \rho(E_0^m) + \rho(E_1) \leq \overline\rho(\llb E_0\rrb) + 1/2k + 1/2k \leq \overline\rho(\llb E\rrb) + 1/k.\qedhere\]
\end{proof}

\begin{Theorem}\label{prop-W2RvsBW2r}
Let $\rho$ be a left-c.e.\ semi-measure.  Then $X \in 2^\omega$ is weakly 2-random for $\rho$ if and only if $X$ is blindly weakly 2-random for $\overline\rho$.
\end{Theorem}

\begin{proof}
For every $E \subseteq 2^{<\omega}$, $\overline\rho(\llb E\rrb) \leq \rho(E)$, and therefore every generalized $\rho$-Martin-L\"of test is also a blind generalized $\overline\rho$-Martin-L\"of test.  Thus if $X$ is blindly weakly 2-random for $\overline\rho$, then $X$ is weakly 2-random for $\rho$.

Conversely, suppose that $X$ is not blindly weakly 2-random for $\overline\rho$. Let $(\MP{U}_n)_{n \in \omega}$ be a blind generalized $\overline\rho$-Martin-L\"of capturing $X$, and let $(E_n)_{n \in \omega}$ be a uniformly c.e.\ sequence of prefix-free subsets of $2^{<\omega}$ such that, for all $n \in \omega$, $\MP{U}_n = \llb E_n\rrb$.  Let $(F_n)_{n \in \omega}$ be the uniformly c.e.\ sequence of prefix-free subsets of $2^{<\omega}$ where $\sigma$ is enumerated in $F_n$ if and only if every $E_i$ with $i \leq n$ enumerates a $\tau_i \preceq \sigma$, and $|\sigma| = \max\{|\tau_i| : i \leq n\} + n$.  Then, for all $n \in \omega$, $\llb F_n\rrb = \bigcap_{i \leq n}\llb E_i\rrb$.  Therefore $X \in  \bigcap_{n \in \omega}\llb E_n\rrb = \bigcap_{n \in \omega}\llb F_n\rrb$.  Furthermore, for all $n \in \omega$, $\rho(F_{n+1}) \leq \rho(F_n)$.  It remains to show that $\lim_{n \imp \infty}\rho(F_n) = 0$.  To see this, observe that, for all $m,n \in \omega$, $\rho(F_{n+m}) \leq \rho(E_n^{n+m})$.  Thus, for all $n \in \omega$,
\begin{align*}
\lim_{m \imp \infty}\rho(F_m) = \lim_{m \imp \infty}\rho(F_{n+m}) \leq \lim_{m \imp \infty}\rho(E_n^{n+m}) = \overline\rho(\llb E_n\rrb),
\end{align*}
where the last equality is by Lemma~\ref{lem-LimitHelper}.  Thus, for all $n \in \omega$, $\lim_{n \imp \infty}\rho(F_n) \leq \overline\rho(\llb E_n\rrb)$.  Since $\lim_{n \imp \infty}\overline\rho(\llb E_n\rrb) = 0$, we must have $\lim_{n \imp \infty}\rho(F_n) = 0$ as well.
\end{proof}

The relationships between the various notions considered here are summarized by the following diagram, where a strict inequality means that there is some semi-measure $\rho$ separating the two notions.

\begin{center}
\begin{tabular}{ccc}
$\wtwor_{\overline\rho}$ & $\subsetneq$ & $\mlr_{\overline\rho}$\\

\rotatebox{270}{$\subsetneq$}   & & \rotatebox{270}{$\subsetneq$}\\

& & \\

$\bwtwor_{\overline\rho}$ & $\subsetneq$  &$\bmlr_{\overline\rho}$\\

\rotatebox{270}{$=$} & & \rotatebox{270}{$\subsetneq$}\\
& & \\

$\wtwor_\rho$ & $\subsetneq$ & $\mlr_\rho$\\
\end{tabular}
\end{center}

\bigskip

\subsection{The No Randomness Ex Nihilo principle for weak 2-randomness with respect to a semi-measure}

As we showed in Section \ref{subsec-w2r}, for each left-c.e.\ semi-measure $\rho$, $\rho$-weak 2-randomness satisfies coherence, randomness preservation, and the computable sequence condition.  The status of the No Randomness Ex Nihilo principle, however, is still open.

\begin{Question}\label{q-ex-nihilo-w2r}
Let $\rho$ be a left-c.e.\ semi-measure.  If $\Phi$ is a Turing functional that induces $\rho$ and $Y\in\wtwor_\rho$, is there some $X\in\wtwor$ such that $\Phi(X)=Y$?
\end{Question}

\noindent A positive answer to Question \ref{q-ex-nihilo-w2r} would also allow us to answer Shen's question for weak 2-randomness, which also remains open.

\begin{Question}
If $\Phi$ and $\Psi$ are Turing functionals such that $\lambda_\Phi(\sigma)=\lambda_\Psi(\sigma)$ for every $\sigma\in\str$, does it follow that $\Phi(\wtwor)=\Psi(\wtwor)$?
\end{Question}

Some partial progress on answering Question \ref{q-ex-nihilo-w2r} has been made.  We show that the No Randomness Ex Nihilo principle holds for weak 2-randomness with respect to any computable measure.

\begin{Theorem} \label{thm:ex-nihilo-w2r}
 Let $\Phi$ be an almost total Turing functional. If $Y\in\wtwor_{\lambda_\Phi}$, there is some $X\in\wtwor$ such that $\Phi(X)=Y$.
\end{Theorem}

\begin{proof}
Let  $(\MP{U}^e_i)_{e,i\in\omega}$ be a (non-effective) listing of all generalized $\lambda$-Martin-L\"of tests.  That is, every generalized Martin-L\"of test is of the form $(\MP{U}^e_i)_{i\in\omega}$ for some~$e$. Without loss of generality we can assume that the first test $(\MP{U}^0_i)_{i\in\omega}$ is the universal Martin-L\"of test. Let $Y\in\wtwor_{\lambda_\Phi}$. Since $Y$ is in particular $\lambda_\Phi$-Martin-L\"of random, by Theorem~\ref{thm:pres-ex-nihilo}, $\Phi^{-1}(Y)\cap\mlr\neq\emptyset$. In other words, for some $i_0$, the pre-image of $Y$ under $\Phi$ meets the $\Pi^0_1$ class $\MP{C}_0=(\MP{U}^0_{i_0})\compl$. We further note that $\Phi$ is total on $\MP{C}_0$. Indeed, $\Phi$ is almost total, which means that $\dom(\Phi)\compl$ has measure~$0$. But $\dom(\Phi)\compl$  is a $\Sigma^0_2$ set, i.e., a union of effectively closed sets, which thus must all have measure~$0$. Since no Martin-L\"of random real can be contained in an effectively closed set of measure~$0$, and since $\MP{C}_0$ contains only Martin-L\"of  random elements, this shows $ \MP{C}_0 \cap \dom(\Phi)\compl = \emptyset$, i.e., $ \MP{C}_0 \subseteq \dom(\Phi)$.

We now build a sequence of non-empty $\Pi^0_1$ classes 
$\MP{C}_1, \MP{C}_2,\dots$ in such a way that
\begin{itemize}
\item $\MP{C}_i\supseteq\MP{C}_{i+1}$ for every $i\geq 0$,
\item for all $n$, all members of $\MP{C}_n$ pass all the tests $(\MP{U}^e_i)_{i\in\omega}$ for $e \leq n$, and
\item for all~$n$, $\Phi^{-1}(Y)\cap\MP{C}_n\neq\emptyset$.
\end{itemize}
Note that since all $\MP{C}_i$ are contained in $\MP{C}_0$, this in particular means that $\Phi$ is total on all~$\MP{C}_i$. Suppose that $\MP{C}_0, \dotsc, \MP{C}_n$ with these properties have already been built. To build $\MP{C}_{n+1}$, we do the following. Suppose that for all~$i$ we have 
\[
Y \in \Phi(\MP{C}_n) \setminus \Phi(\MP{C}_n \cap (\MP{U}^{n+1}_i)\compl).
\]
The pre-image of the set $\Phi(\MP{C}_n) \setminus \Phi(\MP{C}_n \cap (\MP{U}^{n+1}_i)\compl)$ under $\Phi$ is contained in $\MP{U}^{n+1}_i$ and therefore its measure tends to~$0$ as $i$ tends to infinity. By definition of the induced measure $\lambda_\Phi$, this implies that
\[
\lambda_\Phi \left( \Phi(\MP{C}_n) \setminus \Phi(\MP{C}_n \cap (\MP{U}^{n+1}_i)\compl) \right) \rightarrow 0
\]
and thus the set $\bigcap_i (\Phi(\MP{C}_n) \setminus \Phi(\MP{C}_n \cap (\MP{U}^{n+1}_i)\compl))$ is a $\Pi^0_2$ set of $\lambda_\Phi$-measure~$0$ containing~ $Y$, contradicting the fact that $Y$ is $\lambda_\Phi$-weakly 2-random. Thus, there exists~$j$ such that $Y \in \Phi(\MP{C}_n \cap (\MP{U}^{n+1}_j)\compl)$, and we set $\MP{C}_{n+1} = \MP{C}_n \cap (\MP{U}^{n+1}_j)\compl$. This ensures that all elements pass the $(n+1)^{\mathrm{st}}$ generalized Martin-L\"of test. This finishes the construction of the $\MP{C}_n$'s. 

To finish the proof, since $\Phi^{-1}(Y)\cap\MP{C}_i\neq\emptyset$ for every $i\in\omega$, choose $X_i\in\Phi^{-1}(Y)\cap\MP{C}_i$ for each $i$. By the compactness of $\cs$, one can assume, up to extraction of a subsequence, that the sequence $(X_i)_{i\in\omega}$ converges to a limit~$X^*$. For any given~$n$, almost all $i$ are greater than~$n$, and thus $X_{i} \in \MP{C}_{i} \subseteq \MP{C}_n$. Since $\MP{C}_n$ is closed, it implies that the limit $X^*$ belongs to $\MP{C}_n$. This being true for all~$n$, by construction of the $\MP{C}_n$, $X^*$ passes all generalized Martin-L\"of tests, and therefore is weakly 2-random. Moreover, by continuity of Turing functionals on their domain:
\[
\lim_{i \rightarrow \infty}  \Phi(X_{i}) = \Phi(X^*),
\]
but $\Phi(X_{i})=Y$ for all~$i$, therefore $\Phi(X^*)=Y$. This establishes the existence of a weakly 2-random sequence in $\Phi^{-1}(\{Y\})$ and completes the proof.  
\end{proof}

The above proof of Theorem~\ref{thm:ex-nihilo-w2r} is essentially analytic. Let us mention that a completely different proof, of computability-theoretic flavor, can also be given.  Suppose that $Y$ is $\lambda_\Phi$-weakly 2-random. Then by Theorem~\ref{thm-w2r-minimal-pair}, $Y$ does not compute any non-computable~$\Delta^0_2$ set. Let $\MP{C}_0$ be the set defined in the previous proof (on which~$\Phi$ is total) and let $\mathcal{P}=\MP{C}_0 \cap \Phi^{-1}(Y)$, so that $\mathcal{P}\subseteq\mlr\cap\Phi^{-1}(Y)$. It is well-known that given a $\Pi^0_1$ class and a countable collection of reals $(A_i)_{i \in \omega}$, there is a member  of the $\Pi^0_1$ class which does not compute any $A_i$ (Jockusch and Soare~\cite{JockuschS1972} proved this fact for a single~$A$, but it is easy to see that their construction, a forcing argument, can be extended to a countable collection of~$A_i$). Taking the collection $(A_i)_{i\in\omega}$ to consist of the non-computable $\Delta^0_2$ sets, relativizing the previous theorem to~$Y$, and using the fact that each $A_i$ is not $Y$-computable, there exists a member~$X$ of $\mathcal{P}$ which does not compute any~$A_i$. Thus, $X\in\Phi^{-1}(\{Y\})$ is Martin-L\"of random, and does not compute any non-computable $\Delta^0_2$ set. Applying Theorem~\ref{thm-w2r-minimal-pair} again, this shows that $X$ is weakly 2-random.

\bibliographystyle{alpha}
\bibliography{semimeasures}

\end{document}